\documentclass[a4paper,12pt]{article}

\usepackage[bbgreekl]{mathbbol}
\usepackage{mathrsfs}
\usepackage{graphicx}
\usepackage{geometry}
\usepackage{amsmath}
\usepackage{amsfonts}
\usepackage{amssymb}
\usepackage{amsthm}
\usepackage{colortbl}
\usepackage{epstopdf}
\usepackage{color}
\usepackage{bm}
\usepackage{float}
\usepackage{enumerate}
\usepackage[normalem]{ulem}
\usepackage{placeins}
\usepackage{appendix}
\usepackage{textcomp}
\usepackage{tikz}
\usepackage{longtable}
\usepackage{booktabs}
\usepackage{empheq}
\usepackage{subfig}

\newtheorem{theorem}{Theorem}[section]
\newtheorem{lemma}{Lemma}[section]
\newtheorem{remark}{Remark}[section]

\renewcommand{\theequation}{\arabic{section}.\arabic{equation}}

\renewcommand{\thelemma}{\arabic{section}.\arabic{lemma}}

\newcommand\R{\mathbb{R}}
\newcommand\N{\mathbb{N}}
\newcommand\Z{\mathbb{Z}}

\newcommand{\tnmat}{t^{n}_{\ell}(E)}

\newcommand{\lmax}{L}
\newcommand{\sumn}{\sum_{n'=1}^{N}}
\newcommand{\sumnn}{\sum_{n'\ne n}}

\newcommand{\gamn}{\Gamma_{n}}
\newcommand{\omen}{\Omega_{n}}
\newcommand{\omeout}{\Omega^{\rm I}}
\newcommand{\laml}{E_{\lmax}}
\newcommand{\trial}{\mathcal{U}_{\lmax}(E_{\lmax})}
\newcommand{\test}{\mathcal{V}_{\lmax}(E_{\lmax})}

\newcommand{\mollf}{\mathcal{M}}

\newcommand{\ltwo}{L^2(\mathbb{R}^3)}
\newcommand{\hone}{H^1(\mathbb{R}^3)}
\newcommand{\hdel}{H^{\delta}(\mathbb{R}^3)}
\newcommand{\hout}{H^1(\omeout)}
\newcommand{\hn}{H^1(\omen)}

\newcommand{\ujbar}{\bar{u}_L}
\newcommand{\wlz}{\mathcal{W}_L(z)}
\newcommand{\ulz}{\mathcal{U}_L(z)}
\newcommand{\vlz}{\mathcal{V}_L(z)}

\newcommand{\rhat}{\bm{\hat{r}}}
\newcommand{\ylm}{Y_{\ell m}(\hat{\bm{r}})}

\newcommand{\sumlm}{\sum_{\ell m}}
\newcommand{\sumlL}{\sum_{\ell\le L}}
\newcommand{\tl}{t_{\ell}}
\newcommand{\struc}{g_{\ell m, \ell' m'}^{nn'}}
\newcommand{\ul}{u_{\lmax}}
\newcommand{\dd}{{\rm d}}

\newcommand{\vr}{\bm{r}}

\title{Numerical Analysis of the Multiple Scattering Theory for Electronic Structure Calculations\thanks{
X. Li and H. Chen's work were partially supported by National Key R\&D Program of China 2019YFA0709600, 2019YFA0709601. 
H. Chen's work was partially supported by the Natural Science Foundation of China under grant 11971066.
X. Gao's work was partially supported by the Science Challenge Project under Grant TZ2018002.
}
}
\author{Xiaoxu Li
\thanks{
{\tt xiaoxuli@bnu.edu.cn},
College of Education for the Future, Beijing Normal University.
},
~Huajie Chen
\thanks{
{\tt chen.huajie@bnu.edu.cn},
School of Mathematical Sciences, Beijing Normal University.
}
~and Xingyu Gao
\thanks{
{\tt gao\_xingyu@iapcm.ac.cn},
Laboratory of Computational Physics, Institute of Applied Physics and Computational Mathematics.
}
}
\date{}

\begin{document}
\maketitle
% \linenumbers

\begin{abstract}
The multiple scattering theory (MST) is one of the most widely used methods in electronic structure calculations.
It features a perfect separation between the atomic configurations and site potentials, and hence provides an efficient way to simulate defected and disordered systems.
This work studies the MST methods from a numerical point of view and shows the convergence with respect to the truncation of the angular momentum summations, which is a fundamental approximation parameter for all MST methods.
We provide both rigorous analysis and numerical experiments to illustrate the efficiency of the MST methods within the angular momentum representations. 
\end{abstract}

% Section Introduction

\section{Introduction}
\label{sec:intro} 
\setcounter{equation}{0}

Multiple scattering theory (MST) was initiated to study the propagation of heat and electricity through inhomogeneous media \cite{rayleigh92}. It allows a unified treatment of many partial differential equations of modern physics and engineering, and has been applied to
both classical and quantum physics, in particular, the electronic structure calculations of solid states.
Along with the framework of density functional theory (DFT) \cite{hohenberg64, kohn65}, MST has been successfully used to simulate the electronic structure of many systems, including bulk materials \cite{asato99}, surfaces \cite{galanakis02}, impurities \cite{nonas98} and alloys \cite{modinos00}.

The idea of the MST methods is to divide the system into non-overlapping subregions, solve each of the smaller parts separately by quantum scattering theory, and then assembles the partial solutions of the pieces into an overall solution. 
A key feature of the MST method is that it affords a complete separation of the potential aspects from the structure aspects.
Therefore, this approach is extremely potential to develop parallel algorithms for large and complex systems with substitutional defects, for example, the alloys \cite{gonis00,thiess12}.
Another feature of the MST method is that it is a Green's function method.
Instead of solving the eigenvalue problems in standard electronic structure calculations, it can calculate the single-particle Green's function directly. 
In the disordered systems, the statistical average of the Green's function is the right object used to calculate the averaged physical properties, whereas the average of the wave function is not well-defined and cannot be directly related to the physical observables \cite{ebert11,ruban08}.

The idea of MST to calculate the electronic structures was first known as the KKR method \cite{kohn54,korringa47}, formulated by the so-called KKR secular equation, which can determine the eigenvalues and eigenstates of a many-particle system
(see Section \ref{sec:mst}).
Due to the huge computational cost, solving the original KKR equation is not widely used in practice nowadays.
Nevertheless, it still founds the framework of the modern MST methods, by giving the so-called scattering path matrix, which is a key quantity to construct the Green's function in almost all MST methods.
We will only focus on the KKR method in this paper, but our theory (the error estimate of the angular-momentum truncations) is valid for all other practical MST methods.

The MST methods are always formulated and implemented within the angular-momentum representations. 
Therefore the accuracy of the MST solutions depends on the truncation of the angular-momentum summations. 
The convergence and error analysis of the solutions with respect to this truncation are thus crucial, which determine the reliability and efficiency of the MST simulation results.
To our best knowledge, there is no existing work on rigorous mathematical analysis for the MST methods. 
We refer to \cite[Section 5.10]{bulter91} for a non-rigorous discussion.
% based on a perturbation argument
%
The difficulties of analysis for the MST methods lie in two aspects.
First, unlike the traditional discretization methods for electronic structure calculations (such as plane waves \cite{cances12}, finite elements \cite{chen13}, atomic orbitals \cite{chen15b,herring40} and augmented plane waves \cite{chen15a,singh06}), the MST methods are more of a perturbation method rather than a variational method, for which most of the standard approaches in numerical analysis fail. 
Second, the basic KKR secular equation is a highly nonlinear problem, where the coefficients in the equation have very complex dependence on the the energy to be solved.

The purpose of this paper is to derive an {\em a priori} error estimate of the MST solutions (i.e. the solutions of KKR secular equation) with respect to the angular-momentum truncations. 
We only consider linear Schr\"{o}dinger-type eigenvalue problems in this paper, whereas the nonlinear problems deserve to be investigated in a separate work.
The main result is the following superalgebraic convergence rate of the eigenpairs under certain assumptions (see Section \ref{sec:num-anal} for details)
\begin{eqnarray}
|E_{\lmax}-E| + \|u_{\lmax}-u\| \le C_{s} \lmax^{1-s} \qquad\forall~s>1
\end{eqnarray}
where $E$ and $u$ are the energies (eigenvalues) and wave functions (eigenfunctions), $\lmax$ is the angular-momentum cut-off, and the constant $C_{s}$ only depends on the system and $s$.
The convergence result also works for other MST formulations (that are equivalent to the secular equation in math, but different in practice \cite{gonis00}).

\vskip 0.4cm
\noindent{\bf Outline.}
The remainder of this paper is organized as follows. 
In Section \ref{sec:mst}, we give a brief overview of the multiple scattering theory. In particular, we consider the muffin-tin potentials and derive the KKR secular equation for linear Schr\"{o}dinger-type eigenvalue problems within the angular-momentum representations.
In Section \ref{sec:num-anal}, we provide a convergence analysis and an {\em a priori} error estimate of MST approximations (the solutions of the secular equations) with respect to the angular-momentum truncations.
In Section \ref{sec:numerics}, we present some numerical experiments to support the theory.
In Section \ref{sec:conclusions}, we give some concluding remarks and future perspectives.
For the sake of clarity of the presentations, the proofs and some supplementary details are put in the appendices.

\vskip 0.4cm
\noindent{\bf Notations.}
Throughout the paper, we use the symbol $C$ to denote a generic positive constant that may change from one line to the next. 
In the error estimates, $C$ will always remain independent on the discretization parameters or the choice of test functions. The dependence of $C$ on model parameters will normally be clear from the context or stated explicitly.

Our analysis requires a number of notations for domain decompositions and special functions, and we list them in the following table with brief summaries. 
We also list the symbols that are frequently used in the paper.

\begin{longtable}[htbp!]{llll}
\toprule
\bf{General notations} & & \\
\midrule
$\bm{r},~r,~\rhat$ & $\bm{r}\in\R^3$ is the spatial coordinate, $r=|\bm{r}|$ and $\rhat=\bm{r}/r$ & \\
$\bm{r}_{n}$ & $\bm{r}_{n}=\bm{r}-\bm{R}_n$ with $\bm{R}_{n}$ the position of the $n$-th atom \\
$\ell,~m$ & angular-momenta and magnetic quantum numbers & \\
$L$ & angular-momentum truncation & \\
$z,~\sqrt{z}$ & $z$ is an energy parameter,~ $\sqrt{z} := i\sqrt{|z|}$ if $z<0$ & \\
$E,~u(\bm{r})$ & bound state energy and wave function& \\
$E_{L},~u_{L}(\bm{r})$ & approximated bound state energy and wave function& \\[1ex]
$\sumlm,~\sumlL $ & 
$\sumlm := \sum_{\ell=0}^{\infty} \sum_{m=-\ell}^{\ell},~~
\sumlL := \sum_{\ell=0}^{L} \sum_{m=-\ell}^{\ell}$ &\\[1ex]
\midrule
\bf{Domain-related} & \hskip -2.5cm (see Figure \ref{fig-domain}) & \\
\midrule
$\Omega^{\rm{I}}$ & interstitial region & \\
$\bm{R}_{n}$ & position of the $n$-th nucleus & \\
$\Omega_{n},~R_{n},~\Gamma_{n}$ & atomic sphere centered at $\bm{R}_{n}$, its radius, and its surface & \\
$\Omega^{\rm{N}}_{n}$ & a neighbouring region around $\Omega_n$ (see \eqref{ring_domain}) & \\
$B_{n},~r^{\rm b}_n$ & a ring surrounding $\Omega_n$ and its width (see page \pageref{Bn}) & \\
\midrule
\bf{Function-related} & & \\
\midrule
$\ylm$ & spherical harmonic functions &\\
$j_{\ell}(\bm{r}),~n_{\ell}(\bm{r}),~h_{\ell}(\bm{r})$ & spherical Bessel, Neumann, and Hankel functions & \\
$J_{\ell m}(\bm{r};z),~N_{\ell m}(\bm{r};z),~H_{\ell m}(\bm{r};z)$ & $j_{\ell}(\sqrt{z}\bm{r}),~n_{\ell}(\sqrt{z}\bm{r}),~h_{\ell}(\sqrt{z}\bm{r})$ times $\ylm$ (see \eqref{abbr-hankel})\\
$\chi_{\ell}(r;z)$ & solution of the radial Schr\"{o}dinger equation (see \eqref{radial-schrodinger}) & \\
$\zeta^n_{\ell m}(\bm{r}_{n};z)$ & wave function in $\Omega_n$ with energy parameter $z$ (see \eqref{radial-component}) & \\
$V(\bm{r})$ & potential of the system & \\
$V_{0}(r)$ & a spherically symmetric potential centred at origin & \\
\midrule
\bf{MST-related} & & \\
\midrule
$\tl(z)$ & $t$-matrix (see \eqref{t-matrix})\\
$\struc(z)$ & structure constants (see \eqref{structure-constant}) & \\
$\ulz$,~$\vlz$ & functional spaces (see \eqref{trial-space}, \eqref{test-space}) & \\
$\phi_{\ell m,\lmax}^{n}(\bm{r};z)$, $\varphi_{\ell m,\lmax}^{n}(\bm{r};z)$ & basis functions of $\ulz$ and $\vlz$ (see \eqref{trial-basis}, \eqref{test-basis}) & \\
$\mollf$ & a map from $\ulz$ to $\vlz$ (see \eqref{mollify})& \\
$\eta_n(r)$ & a cut-off function on $B_n$ (see \eqref{cut-off-eq})& \\
\bottomrule
\end{longtable}
%
%\begin{figure}[htbp]
%	\centering
%	\includegraphics[width=0.8\textwidth]{figs/decomposition.png}
%	\caption{Decomposition of the space into the atomic spheres and interstitial region.}
%	\label{fig-domain}	
%\end{figure}
%
\begin{figure}[htb]
	\begin{center}
	\begin{tikzpicture}[scale=0.8]
	\draw [loosely dashed, line width=2pt, fill=gray!20] (1,1) circle(3.7);
	\draw [line width=2pt, red!70, fill=white] (1,1) circle (1);
	\filldraw (1,1) circle (1.5pt);
	\node at (1.4,1) {$\bm{R}_{1}$};
	\node at (0.1,1) {$\Gamma_{1}$};
	\node at (1,0.4) {$\Omega_{1}$};
	\node [scale=1.5] at (0.8,-1.5) {$\Omega^{N}_{1}$};
		
	\draw [densely dotted, line width=2pt, fill=gray!50] (4,3) circle(2);
	\draw [line width=2pt, red!70, fill=white] (4,3) circle (1.3);
	\filldraw (4,3) circle (1.5pt);
%	\draw [ultra thick] (4,4.3)--(4,4.8);		
%	\node [scale=1.2] at (4.4,4.4) {$r^{\rm b}_{2}$};
	\node at (4.4,3) {$\bm{R}_{2}$};
	\node at (2.8,3) {$\Gamma_{2}$};
	\node at (4,2.4) {$\Omega_{2}$};
	\node at (4,1.3) {$B_{2}$};
		
	\draw [densely dotted, line width=2pt, fill=gray!50] (-4,4) circle(2.7);
%	\draw [line width=2pt, green!50, fill=white] (-4,4) circle (2);
	\draw [line width=2pt, red!70, fill=white] (-4,4) circle (2);
%	\draw [ultra thick] (-6.5,4)--(-6,4);
%	\node [scale=1.2] at (-6.1,4.3) {$r^{\rm b}_{3}$};
	\filldraw (-4,4) circle (1.5pt);
	\node at (-3.6,4) {$\bm{R}_{3}$};
	\node at (-5.9,4) {$\Gamma_{3}$};
	\node at (-4,3.4) {$\Omega_{3}$};
	\filldraw (-4,4) circle (1.5pt);
	\node at (-4,1.6) {$B_{3}$};
		
	\node [scale=1.5] at (-4,0) {$\Omega^{\rm{I}}$};
	\end{tikzpicture}
	\caption{Decomposition of the space into the atomic spheres and interstitial region.}
	\label{fig-domain}
	\end{center}
\end{figure}
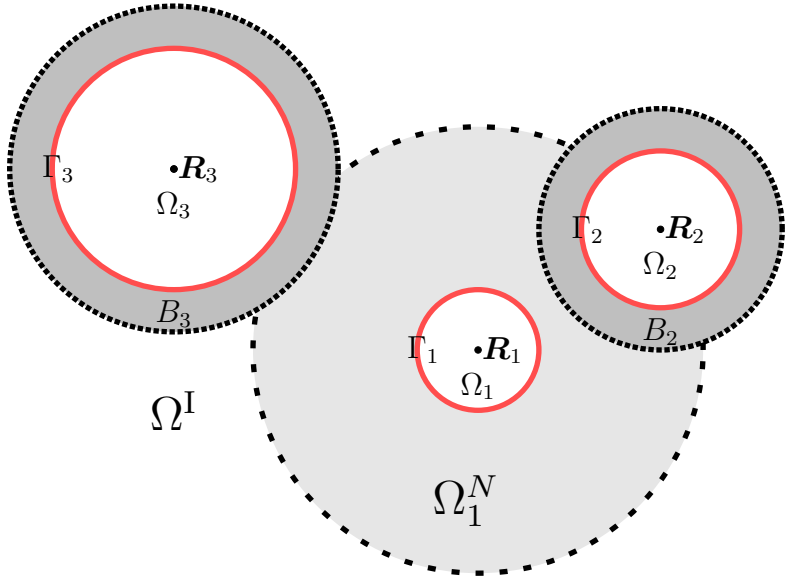

% Section MST

\section{Multiple scattering theory}
\label{sec:mst}
\setcounter{equation}{0}

In this section, we will apply multiple scattering theory to a Schr\"{o}dinger-type equation and derive the KKR secular equation. 
Consider the following linear eigenvalue problem
\begin{eqnarray}
\label{schrodinger}
\Big(-\Delta + V({\bm r})\Big) u({\bm r}) = E u({\bm r}) \qquad \bm{r}\in\mathbb{R}^3,
\end{eqnarray}
where $V$ is the potential of the system,
$u$ is subject to the boundary condition $|u({\bm r})|\rightarrow 0$ as $|\vr|\rightarrow \infty$ (such that it is a bound state), and $E$ is the corresponding energy.
In the following, we will focus on the case when $V$ is the so-called muffin-tin potentials, which are confined in non-overlapping spheres $\Omega_n~(n=1,\cdots,N)$ centered at atomic sites $\bm{R}_n$ with radius $R_n$.
%, and the boundary of $\Omega_n$ is denoted by $\Gamma_n$ (see Figure \ref{fig-domain} for the domain decomposition).
%
More precisely, the potential takes the form
\begin{eqnarray}
\label{mt-potential}
V({\bm r}) = \sum_{n=1}^{N} v_{n}\big(|{\bm r}-{\bm R}_n|\big) \qquad \bm{r}\in\mathbb{R}^3,
\end{eqnarray}
with $v_n(r) \in C^{\infty}(0,\infty)$ and $v_n(r)=0$ when $r> R_n$. 
Here the potential has been assumed to be spherically symmetric within each atomic sphere only for simplicity of presentations. 
The analysis in this paper can be easily extended to general potentials that are non-spherically symmetric in the atomic spheres (see \ref{append-asym}).

The confinement that the potential is supported in atomic spheres is necessary for our analysis, which  greatly simplifies the expressions of free electron propagator in the angular-momentum expansions. 
From a practical point of view, the MST algorithms can be generalized to arbitrary potentials by replacing the atomic spheres $\Omega_n$ with the so-called space-filling cells (e.g. the Voronoi cell decompositions) \cite{gonis00}, for which the mathematical formulations and numerical analysis are far more complicated. 
We will study the general cases (without interstitial region) in a separate work.
We remark that the muffin-tin potential approximations not only allow dramatic simplifications for theoretical analysis and numerical algorithms, but also make very successful applications to many physical systems like metals and alloys \cite{martin05,morruzi78} where the use of the theory has continued until the present.

In our analysis, we will need to assume some regularity of the potential: 
$V\in L^{2}(\mathbb{R}^3)$ and $V|_{\Omega_n}\in C^{\infty}((0,R_n)\times\mathbb{S}^2)$.
%$ \bigcap$ $\left( \bigcap_{n=1}^N C^{\infty}(\Omega_n\backslash \{\vR_{n}\}) \right)$
%
This is a mild assumption that can be satisfied in most of the physical problems of practical interests.
For example, if the muffin-tin potential has singular Coulomb potential around the atomic sites, say $v_n(r)\approx C r^{-1}$ near the origin, then the total potential $V$ satisfies this condition.
It then follows from \cite[p.119]{reed72} that $-\Delta+V$ has continuous spectrum on $[0,+\infty)$ and discrete spectrum below zero.
We emphasize that we only consider the discrete spectrum $E<0$ in this paper, which corresponds to the bound states of the system \eqref{schrodinger}. 

The approach of multiple scattering for solving \eqref{schrodinger} is usually formulated by two steps: 
(i) determining the relationship between the free electron gas and the system perturbed by a single atomic potential, and (ii) deriving a secular equation that connects the amplitudes of incoming and outgoing waves at multiple atomic sites.

\subsection{Single-site scattering}
\label{subsec:single-scatter}

Consider the following single-site scattering problem with a fixed energy parameter $z<0$
\begin{eqnarray}
\label{single-schrodinger}
\Big(-\Delta + V_{0}(r)-z\Big) \psi({\bm r})= 0 \qquad \bm{r}\in\mathbb{R}^3,
\end{eqnarray}
where the potential is spherically symmetric and can be written by a radial function $V_{0}$ satisfying $V_{0}(r)=0$ for $r\geq R_0$.
Note that no (far-field) boundary condition is imposed on this equation, since we only concern the relationship between the incoming and outgoing waves by viewing the equation as a single potential scattering procedure.
The solution $\psi$ depends on parameter $z$, but we suppress this dependency throughout this paper for simplicity of presentations.

The solution $\psi$ can be written in the angular-momentum representation
\begin{eqnarray}
\label{radial-representation}
\psi({\bm r})=\sumlm c_{\ell m} \chi_{\ell}(r;z) \ylm,
\end{eqnarray}
where $Y_{\ell m}$ is the spherical harmonic functions \cite{helgaker00}
and the radial part $\chi_{\ell}$ satisfies the following radial Schr\"{o}dinger equation at energy parameter $z$
\begin{subequations}
	\label{radial-schrodinger}
	\begin{empheq}[left=\empheqlbrace]{align}
	\label{radial-schrodinger-a}
	\displaystyle\left(\frac{\partial^2}{\partial r^2} + \frac{2}{r} \frac{\partial }{\partial r} -  \frac{\ell(\ell+1)}{r^2} -V_{0}(r) +z \right) \chi_{\ell}(r;z)=0 
	& \qquad  0<r<R_0 ,
	\\[1ex]
	\label{radial-schrodinger-b}
	\displaystyle\left(\frac{\partial^2}{\partial r^2} + \frac{2}{r} \frac{\partial }{\partial r} -  \frac{\ell(\ell+1)}{r^2} +z \right) \chi_{\ell}(r;z)=0
	& \qquad r>R_0 .
	\end{empheq}
\end{subequations}

We first study \eqref{radial-schrodinger-b} outside the atomic sphere with vanishing potential, which is the well-known spherical Bessel differential equation \cite{gil07}. 
For given $\ell\in\mathbb{N}$ and $z<0$, the solutions could be the spherical Bessel functions $j_{\ell}(\sqrt{z}r)$, the spherical Newmann functions $n_{\ell}(\sqrt{z}r)$, or the spherical Hankel functions $h_{\ell}(\sqrt{z}r)=j_{\ell}(\sqrt{z}r)+i n_{\ell}(\sqrt{z}r)$.
The differences between these three special functions lie in their asymptotic behavior at the origin and infinity.
For $z<0$, the spherical Bessel functions $j_{\ell}(\sqrt{z}r)$ are regular at the origin while the other two diverge, and the spherical Hankel functions $h_{\ell}(\sqrt{z}r)$ are bounded at infinity while the other two are not (see, e.g. \cite{gil07,gonis00}).
Since any two of the above functions can make a complete basis set for the solutions of \eqref{radial-schrodinger-b}, the radial solutions outside the atomic sphere can be represented by a linear combination of the spherical Bessel and Hankel functions
\begin{eqnarray}\nonumber
\label{radial-general}
\label{chi_l}
\chi_{\ell}(r;z) = a_{\ell}(z) j_{\ell}(\sqrt{z}r) + b_{\ell}(z) h_{\ell}(\sqrt{z}r) \qquad r\geq R_0 .
\end{eqnarray}
From a physical point of view, $\chi_{\ell}$ can be viewed as a superposition of an incoming wave and a scattering wave, which correspond to the two parts with Bessel and Hankel functions, respectively.

The goal of single-site scattering is to determine the relationship between these two parts, more precisely, the ratio between the coefficients $b_{\ell}$ and $a_{\ell}$
\begin{eqnarray}
\label{t-matrix}
t_{\ell}(z) := \frac{b_{\ell}(z)}{a_{\ell}(z)} 
\qquad \rm{for}~\ell\in\mathbb{N}.
\end{eqnarray}
This gives the so-called $t$-matrix (it is called a ``matrix" since it is indexed by the angular momentum $\ell$ for a given atomic site, and there are off-diagonal terms for non-spherical potential, see \ref{append-asym}) in MST method and is used as a fundamental building block of the whole theory. 

From the numerical intuition, the $t$-matrix can be determined by matching the values and derivatives of the radial solution $\chi_{\ell}$ at $r=R_0$.
Since multiplying $\chi_{\ell}$ by an arbitrary constant does not change the solution of \eqref{radial-schrodinger}, it is only required to match the ``logarithm derivatives" $\big(\log(\chi_{\ell})\big)'$ at $r=R_0$: 
\begin{eqnarray}
\label{log-match}
\frac{\chi_{\ell}(r;z)}{\chi'_{\ell}(r;z)} \bigg|_{r=R^-_0} = \frac{a_{\ell}(z) j_{\ell}(\sqrt{z}r) + b_{\ell}(z) h_{\ell}(\sqrt{z}r)}{\sqrt{z} a_{\ell}(z) j'_{\ell}(\sqrt{z}r) + \sqrt{z} b_{\ell}(z) h'_{\ell}(\sqrt{z}r)}\bigg|_{r=R^+_0}.
\end{eqnarray}
Then a direct calculation with the definition \eqref{t-matrix} leads to
\begin{eqnarray}
\label{log-deriv}
t_{\ell}(z)=\frac{\sqrt{z} \chi_{\ell}(R_0;z) j'_{\ell}(\sqrt{z}R_0) -  \chi'_{\ell}(R_0;z) j_{\ell}(\sqrt{z}R_0)}
{\chi'_{\ell}(R_0;z) h_{\ell}(\sqrt{z}R_0) - \sqrt{z} \chi_{\ell}(R_0;z) h'_{\ell}(\sqrt{z}R_0)} ,
\end{eqnarray}
where we have denoted by $\chi'_{\ell}(R_0;z):=\frac{\partial}{\partial r}\chi_{\ell}(r;z)\big|_{r=R^-_{0}}$ the left-hand derivative of $\chi_{\ell}$ at $r=R_0$.
Note that the undetermined multiplicative constant involved in $\chi_{\ell}$ can be canceled in \eqref{log-deriv}, which leads to a unique $t_{\ell}(z)$.

In the above matching procedure, a so-called ``asymptotic problem" \label{asymptotic_problem} (see e.g. \cite{chen15a,singh06}) with $\chi_{\ell}(R_0;z)=0$ must be avoided such that the conditions \eqref{log-match} and \eqref{log-deriv} can make sense.
We will specify an explicit assumption to avoid this ``asymptotic problem" later in Section \ref{sec:cutL}.

An alternative way to derive the $t$-matrix is to use the so-called Lippmann-Schwinger equation \cite{gonis00,lippmann50,martin05}, which is the most widely-used approach in the quantum scattering theory.
The Lippmann-Schwinger equation gives the relationship between incoming and outgoing waves via the free electron Green's function, 
and writes the $t$-matrix as
\begin{eqnarray}
\label{t_l}
t_{\ell}(z)=-i\sqrt{z} \int_{0}^{R_0}j_{\ell}(\sqrt{z}r) V_{0}(r) \chi_{\ell}(r;z) r^2 \dd r ,
\end{eqnarray}
where $\chi_{\ell}$ is the solution of \eqref{radial-schrodinger-a} with some normalization condition.
For the sake of completeness, we give a detailed explanation of \eqref{t_l} in \ref{append-free-electron-Green},
and provide a rigorous analysis to show that the two expressions \eqref{log-deriv} and \eqref{t_l} for $t$-matrix are equivalent.
The integration form \eqref{t_l} from the Lippmann-Schwinger equation can be more easily generalized to non-spherical potentials (see \ref{append-asym}) and even arbitrary space-filling potentials (see \cite{gonis00}), which is therefore more widely used in practice.
However, our analysis relies on the understanding of the matching condition \eqref{log-deriv}.

\subsection{Multiple-site scattering }
\label{subsec:multi-scatter}

The MST methods consider the eigenfunctions of \eqref{schrodinger} in the interstitial region as a wave scattered by the muffin-tin potential.
For each atomic site, the incoming wave is expressed as the scattering waves from all other sites, while the resulting outgoing wave satisfies the single-site scattering relationship in the previous section.
We will then briefly review how to derive the basic secular equation within this MST framework.

We first denote by
\begin{eqnarray}
\label{abbr-hankel}
H_{\ell m}(\bm{r};z)=h_{\ell}\big(\sqrt{z}r\big) \ylm \quad
{\rm and} \quad J_{\ell m}(\bm{r};z)=j_{\ell}\big(\sqrt{z}r\big) \ylm 
\end{eqnarray}
for $\ell\in\N$ and $|m|\leq \ell$.
These functions give a complete basis set for the free-electron system with energy parameter $z$.
Then the eigenfunction $u$ of \eqref{schrodinger} (with corresponding eigenvalue $E$) in the interstitial region can be written as a sum of scattering waves from all scattering center ${\bm R}_{n}$,
\begin{eqnarray}
\label{multi-wavefun}
u(\bm r) =\sum_{n=1}^N \sumlm b^{n}_{\ell m} H_{\ell m}(\bm{r}_{n};E) \qquad \bm{r}\in \Omega^{\rm{I}}
\end{eqnarray}
with the coefficients $b^{n}_{\ell m}$ and local coordinates ${\bm r}_{n}:={\bm r} - {\bm R}_{n}$.
We only use the Hankel functions for the expansion \eqref{multi-wavefun} since $h_{\ell}(\sqrt{E}r)$ is bounded at infinity for $E<0$ while $j_{\ell}(\sqrt{E}r)$ is not, so that the far-field boundary condition of \eqref{schrodinger} can be satisfied.

Alternatively, we can represent the eigenfunction as a superposition of incoming and outgoing waves around a specific atomic sphere.
More precisely, for a given site $\bm R_{n}$, we write
\begin{eqnarray}
\label{single-wavefun}
u(\bm r) =  \sumlm
\big(a^{n}_{\ell m} J_{\ell m}(\bm{r}_{n};E) + b^{n}_{\ell m} H_{\ell m}(\bm{r}_{n};E)\big) \qquad \bm{r}\in\Omega^{\rm{N}}_{n}
\end{eqnarray}
with $\Omega^{\rm{N}}_{n}\subset\Omega^{\rm I}$ the around $n$-th atomic sphere (see Figure \ref{fig-domain})
\begin{eqnarray}
\label{ring_domain}
\Omega^{\rm{N}}_{n} = \Big\{ {\bm r}\in\Omega^{\rm I} ~\Big|~R_{n}<|{\bm r}_{n}|< \min_{n'} \{|\bm{R}_{n}-\bm{R}_{n'}| \} \Big\}.
\end{eqnarray}

We then have a so-called ``expansion theorem" that can bridge the two expressions \eqref{multi-wavefun} and \eqref{single-wavefun} for the eigenfunction $u({\bm r})$ in the interstitial region.
For $n'\neq n$, the Hankel functions centered at atomic site $\bm{R}_{n}$ can be expanded by the Bessel functions centered at $\bm{R}_{n'}$
\begin{eqnarray}
\label{expansion-h-j}
H_{\ell m}(\bm{r}_n;E)=\sum_{\ell' m'} \struc(E) J_{\ell' m'}(\bm{r}_{n'};E) \qquad \vr\in\Omega^{\rm{N}}_{n'},
\end{eqnarray}
where the coefficients $\struc$ are the ``structure constants" depending only on the energy parameter $E$ and the displacement between two atomic sites $\bm{R}_{n'}-\bm{R}_{n}$
\begin{eqnarray}
\label{structure-constant}
\struc(E)=4\pi \sum_{\ell '' m''} i^{l-l'-l''} C_{\ell m, \ell' m', \ell''m''} H_{\ell'' m''}(\bm{R}_{n'}-\bm{R}_{n};E) \qquad n\ne n'
\end{eqnarray}
with $\displaystyle C_{\ell m, \ell' m', \ell''m''} = \int_{S^2} Y_{\ell m}(\hat{\bm{r}}) Y_{\ell' m'}(\hat{\bm{r}}) Y_{\ell''m''}(\hat{\bm{r}}) \dd\hat{\bm{r}}$ being the Gaunt coefficients \cite{mavropoulos06}.
By convention, the structure constant is set to be $0$ for $n=n'$.
Note that the expansion \eqref{expansion-h-j} holds only in the region $\Omega^{\rm{N}}_{n'}$, which is the reason we introduce this ``neighbouring domain".
We refer to \cite[Appendix D]{gonis00} for detailed discussions of the expansion \eqref{expansion-h-j}.

Combining \eqref{multi-wavefun}, \eqref{single-wavefun}, \eqref{expansion-h-j} and the symmetry of $\struc(E)$ (i.e. $\struc(E)=g_{\ell' m', \ell m}^{n'n}(E)$), we have that for any atomic site $\pmb{R}_n$,
\begin{eqnarray}
\label{coef-in-sc}
a^{n}_{\ell m}=\sumnn \sum_{\ell'm'} \struc(E) b^{n'}_{\ell'm'}.
\end{eqnarray}
Moreover, for each atomic site $\pmb{R}_n$, the ratio between the coefficients $a^{n}_{\ell m}$ and $b^{n}_{\ell m}$ has been given by the $t$-matrix \eqref{t-matrix} of single-site scattering in the previous section
\begin{eqnarray}
\label{multi-incoming-scattering}
b^{n}_{\ell m} = \tnmat a^{n}_{\ell m}.
\end{eqnarray} 
Taking into accounts \eqref{coef-in-sc} and \eqref{multi-incoming-scattering}, we have the following equations
\begin{multline}
\label{secular-eq}
\qquad \sumn \sum_{\ell'm'} \Big(\delta_{\ell \ell'} \delta_{m m'} \delta_{nn'} - t^{n'}_{\ell'}(E) \struc(E) \Big) a^{n'}_{\ell'm'} = 0
\\
{\rm for}~ n = 1,\cdots,N,~\ell=0,1,\cdots,~{\rm and}~-\ell\leq m\leq\ell . \qquad
\end{multline} 
In order to have non-trivial solutions of the system \eqref{secular-eq}, we require $E$ to take the values such that the determinant of the coefficient matrix vanishes.
This gives us the eigenvalues $E$ of the problem \eqref{schrodinger}.
Note that although \eqref{secular-eq} is derived for spherically symmetric potentials in this section, it can be easily generalized to non-spherical potentials within the muffin-tin form (see \ref{append-asym}).

\subsection{The angular-momentum cut-off}
\label{sec:cutL}

Since \eqref{secular-eq} is in principle an infinite dimensional problem, one needs to truncate $\ell$ to obtain a finite dimensional approximation in practical calculations.
If the angular quantum number $\ell$ is truncated by some cut-off $L$, i.e. $\ell\leq L$, then the solutions of \eqref{secular-eq} is approximated by $E_L$ satisfying
\begin{eqnarray}
\label{kkr}
{\rm Det} \big( S(E_L) \big) = 0,
\end{eqnarray}
where the matrix $S(E_L)\in \mathbb{C}^{(N(L+1)^2)\times (N(L+1)^2)}$ is ``truncated" from \eqref{secular-eq} with matrix elements
\begin{eqnarray}\nonumber
S_{n\ell m,n'\ell'm'}(E_L) = \delta_{\ell \ell'} \delta_{m m'} \delta_{nn'} - t^{n'}_{\ell'}(E_{L}) \struc(E_{L}).  
\end{eqnarray}
This is the so-called KKR secular equation that lays the foundation of MST methods.
It is similar to an eigenvalue problem, in the sense that the ``approximate eigenvalue" $E_L$ yields a singular matrix $S(E_L)$.
However, \eqref{kkr} is very difficult to solve since the dependence of the matrix on $E_L$ is highly nonlinear.
One might resort to some ``root-trace" algorithm (which is used in our numerical experiments, see Section \ref{sec:numerics} for details) to find the solution.

With the approximate solution $E_L$, the corresponding wave function inside each atomic sphere $\Omega_n$ is given by 
\begin{eqnarray}
\label{radial-component}
\zeta^n_{\ell m}(\bm{r}_{n};E_L) = \chi^n_{\ell}(r_{n};E_L) Y_{\ell m}(\bm{\hat{r}}_n) \qquad |\vr_n|<R_n,
\end{eqnarray}
where $\chi_{\ell}^n$ is the solution of \eqref{radial-schrodinger-a} satisfying the boundary condition
\begin{eqnarray*}
\chi^{n}_{\ell}(R_n;E_{\lmax}) = j_{\ell}(\sqrt{E_{\lmax}}R_n) + t^{n}_{\ell}(E_{\lmax}) h_{\ell}(\sqrt{E_{\lmax}}R_n) .
\end{eqnarray*}
We then write the wave function as
\begin{eqnarray}
\label{kkr-wavefun}
u_{\lmax}(\bm r;E_{\lmax}) = \left\{
\begin{array}{ll}
\displaystyle \sum_{\ell\le \lmax} a^{n}_{\ell m}(E_{\lmax}) \zeta^{n}_{\ell m}(\bm{r}_{n};E_{\lmax})
&\quad \bm{r}\in \omen~(n=1,\cdots, N), 
\\[1ex]
\displaystyle \sum_{n=1}^N \sum_{\ell \le \lmax} b^{n}_{\ell m}(E_{\lmax}) H_{\ell m}(\bm{r}_{n};E_{\lmax}) 
&\quad \bm{r}\in \Omega^{\rm I},
\end{array}
\right.
\end{eqnarray}
where the coefficients $\pmb{a}:=\{a^{n}_{\ell m}(E_{\lmax})\}$ satisfies the equations
\begin{eqnarray}
\label{eq:alm}
S(E_L) \pmb{a} = 0
\end{eqnarray}
with some normalization condition (for example, we can simply take $\|\pmb{a}\|_{\ell^2}=1$),
and $b^{n}_{\ell m}(E_L) = t^{n}_{\ell}(E_{L}) a^{n}_{\ell m}(E_L)$.

Throughout our analysis, we will make the following assumption
\begin{eqnarray}
\label{assumption}
{\rm Det} \big( P(E_L) \big) \ne 0,
\end{eqnarray}
where the matrix $P(E_L)\in \mathbb{C}^{(N(L+1)^2)\times(N(L+1)^2)}$ has elements
\begin{eqnarray}
\label{assump-inv}
P_{n\ell m,n'\ell'm' }(E_L) =  \delta_{nn'} \delta_{\ell\ell'} \delta_{mm'} + \frac{j_{\ell'}(\sqrt{E_L}R_{n'})}{h_{\ell'}(\sqrt{E_L}R_{n'})} \struc(E_L).
\end{eqnarray}
This condition can make sure that the ``asymptotic problem" $\chi^{n}_{\ell}(R_n;z)= 0~(n=1,\cdots,N, \ell\le L$, see page \pageref{asymptotic_problem}) does not happen at energy parameter $z=E_L$.
Note that if $\chi^{n}_{\ell}(R_n;E_L)=0$, then \eqref{log-deriv} implies $t^n_{\ell}(E_L) = -j_{\ell}(\sqrt{E_{\lmax}}R_n)/h_{\ell}(\sqrt{E_{\lmax}}R_n)$, which makes 
$S(E_L)=P(E_L)$ and leads to a contradiction between the KKR secular equation \eqref{kkr} and the assumption \eqref{assumption}.

The angular-momentum truncation $L$ is the only numerical discretization parameter we have introduced so far, which is also the fundamental parameter for almost all MST methods.
The goal of this work is to estimate the error from the truncation of angular-momentum sums. 
It has been observed in practical simulations that the truncation at relatively small number, say $L=2$, $L=3$, or $L=4$, can already give satisfactory results in many cases of physical interest, such as that of metallic systems on close-packed lattices \cite{gonis00}.
The fast convergence with respect to $L$ derived in the following section can provide a theoretical justification for this phenomenon.

% analysis

\section{Numerical analysis of the MST approximations}
\label{sec:num-anal} 
\setcounter{equation}{0}

In this section, we will present an error estimate for the solutions of \eqref{kkr} with respect to the angular-momentum cut-off $\lmax$. 
The analysis consists of three steps.
First, we characterize the (approximate) solution space of the KKR secular equation, in particular, we give the weak continuity condition of \eqref{kkr-wavefun} through atomic spherical surfaces.
Then we convert the KKR secular equation into an equivalent variational form, whose solutions are exactly the same as those of KKR secular equation. 
Finally we provide a convergence analysis for the equivalent variational form and obtain a spectral convergence rate.

We consider the following weak form of eigenvalue problem \eqref{schrodinger}: 
Find $E\in \mathbb{R}$ and $u\in \hone$ such that $\|u\|_{\ltwo}=1$ and 
\begin{eqnarray}\label{model-weak-eq}
\label{model-eq}
a(u,v)=E(u,v)\qquad\forall~v\in \hone,
\end{eqnarray}
where the bilinear form $a(\cdot, \cdot): \hone\times \hone\rightarrow\mathbb{C}$ is defined by
\begin{eqnarray*}
a(u,v) := \int_{\mathbb{R}^3} \nabla u\cdot \nabla v^* + V u v^*
\end{eqnarray*}
with asterisk denoting the conjugate.
This bilinear form is coercive on $H^{1}(\mathbb{R}^3)$, i.e., there exist positive constants $\alpha$ and $\beta$ such that 
\begin{eqnarray}
\label{a-coercive}
	a(v,v) \ge \alpha \|v\|^{2}_{H^1(\mathbb{R}^3)} - \beta \|v\|^{2}_{L^2(\mathbb{R}^3)} \qquad \forall~v\in H^1(\mathbb{R}^3).
\end{eqnarray}
For simplicity, we will assume $\beta=0$ afterwards since a shifted bilinear form $a^{\beta}(u,v):=a(u,v) + \beta(u,v)$ makes no essential change to the eigenvalue problem.

\subsection{Approximate wave function space}
\label{subsec:weak-continuity}

We define the following space related to the muffin-tin decomposition
\begin{eqnarray*}
H^{\delta}(\mathbb{R}^3) = \left\{ v\in L^2(\mathbb{R}^3) ~\Big| ~v|_{\Omega^{\rm I}}\in H^{1}(\Omega^{\rm I}), ~v|_{\omen}\in H^{1}(\omen)~~ \forall~1\le n\le N \right\} ,
\end{eqnarray*}
with a broken Sobolev norm $\|v\|_{\delta} = \|v\|_{H^{1}(\Omega^{\rm I})} + \sum_{n=1}^{N} \|v\|_{H^{1}(\omen)}$,
We further define a bilinear form 
$a_{\delta}(\cdot,\cdot):  H^{\delta}(\mathbb{R}^3) \times  H^{\delta}(\mathbb{R}^3) \rightarrow \mathbb{C}$ by
\begin{eqnarray}\nonumber
	a_{\delta}(u,v) := \int_{\Omega^{\rm I}} \nabla u \cdot\nabla v^* + \sum_{n=1}^{N} \int_{\omen} \nabla u \cdot\nabla  v^* + \int_{\mathbb{R}^3} V u v^*.
\end{eqnarray}
Inherited from the coercivity of $a(\cdot,\cdot)$ in \eqref{a-coercive}, we have that there exist $\alpha_{\delta}>0$ such that 
\begin{eqnarray}\label{coerciveness}\nonumber
	a_{\delta}(v,v) \ge \alpha_{\delta} \|v\|^{2}_{\delta} \qquad \forall~v\in H^{\delta}(\mathbb{R}^3).
\end{eqnarray}

We introduce the notion of ``weak continuity" for the functions in $H^{\delta}(\mathbb{R}^3)$ that are continuous through the atomic spherical surfaces $\Gamma_n~(n=1,\cdots,N)$ for low angular-momentum components.
For $v\in H^{\delta}(\mathbb{R}^3)$, we denote by $[v]_{\Gamma_n} := v^{+}|_{\gamn}-v^{-}|_{\gamn}$ the jump through surface $\gamn$, 
where $v^{\pm}|_{\gamn}$ denote the traces of $v$ taken from inside and outside the sphere $\omen$, respectively.
For a given $L\ge 0$, we say that $v$ is weakly continuous through $\gamn$ if
\begin{eqnarray}
\label{mortar-value}
\int_{\gamn} [v]_{\gamn}  Y_{\ell m}(\hat{\bm{r}}) \dd\hat{\bm{r}}=0 
\qquad \forall ~0\le \ell\le L, ~|m| \le \ell .
\end{eqnarray}

We then consider the following finite dimensional space with given angular-momentum truncation $L$ and energy parameter $z$
\begin{multline}
\label{solution-space}
\wlz = \Big\{ \psi\in H^{\delta}(\mathbb{R}^3) ~\Big|~
\psi|_{\Omega^{\rm I}} = \sum_{n=1}^N\sum_{\ell\leq\lmax} \tilde{c}_{n\ell m} H_{\ell m}(\bm r_{n};z), ~
\\[1ex]
\psi|_{\omen} = \sum_{n=1}^{N}\sum_{\ell\leq\lmax} \bar{c}_{n\ell m} \zeta^{n}_{\ell m}(\bm r_{n};z) , ~
\psi~{\rm and}~\frac{\partial \psi}{\partial r_n} ~\text{satisfy \eqref{mortar-value} } {\rm for}~  n=1,\cdots,N
\Big\} , \qquad
\end{multline}
where $\tilde{c}_{n\ell m}$ and $\bar{c}_{n\ell m}$ are coefficients, $\zeta^{n}_{\ell m}$ is given by \eqref{radial-component},
and $\displaystyle \partial / \partial r_n$ denotes the derivative along the radial direction with respect to the $n$-th atomic site $\bm{R}_{n}$.
It is obvious that the functions in $\wlz$ satisfy the following strong form in the interstitial region and atomic spheres, respectively
\begin{eqnarray}
\label{bvp-schrodinger}
(-\Delta + V({\bf r})-z) u_L({\bm r};z)=0 
\displaystyle \qquad \bm{r}\in \mathbb{R}^{3} \backslash \left(\bigcup_{n=1}^N\gamn\right) .
\end{eqnarray}
The solutions $u_L$ inside and outside the spheres are then patched together through $\gamn$ by the weak continuity conditions \eqref{mortar-value}.

Note that the space $\wlz$ is trivial (only contains zero function) for most of the energy parameter $z$, and non-trivial only for the solution $z=E_L$ of the KKR secular equation \eqref{kkr}.
This is justified in the following lemma, whose proof is given in \ref{append-weak-continuity}.

\begin{lemma}
\label{lem-weak-continuous}
$E_{L}$ solves \eqref{kkr} if and only if $\mathcal{W}_L(E_L) \backslash \{0\} \ne\emptyset$.
Moreover, if $E_{L}$ solves \eqref{kkr}, then
$\mathcal{W}_L(E_L) = \big\{u\in H^{\delta}(\R^3)~\big|~u\text{ can be written as \eqref{kkr-wavefun}}\big\}$.
\end{lemma}

\subsection{Equivalent variational form}
\label{subsec:galerkin-form}

In this subsection, we will construct an eigenvalue problem in a variational form, whose eigenvalues can be shown to be the same as the solutions of the KKR secular equation \eqref{kkr}.

For a given energy parameter $z$, we define the trial function space
\begin{eqnarray}
\label{trial-space}
\ulz = \left\{u\in H^{\delta}(\mathbb{R}^3) ~\Big|~ u = \sum_{n=1}^N\sum_{\ell\leq\lmax} c_{n\ell m} \phi_{\ell m,\lmax}^{n}(\bm{r};z) \right\}, 
\end{eqnarray}
%with dimension $N\times(L+1)^2$ and 
where the basis function $\phi_{\ell m,\lmax}^{n}$ is given by
\begin{eqnarray}
\label{trial-basis}
\phi_{\ell m,\lmax}^{n}(\bm{r};z) = \left\{
\begin{array}{ll}
H_{\ell m}({\bm r_n};z) \quad & {\rm in}~\Omega^{\rm I} ,
\\[1ex]
h_{\ell}(\sqrt{z}R_n) \zeta_{\ell m}^{n}(\bm{r}_n;z) \quad &{\rm in} ~\omen ,
\\[1ex]
\displaystyle \sum_{\ell'\le L} \Big(\struc(z) j_{\ell'}(\sqrt{z}R_{n'}) \Big) \zeta_{\ell'm'}^{n'} (\bm{r}_{n'};z) \quad 
&{\rm in}~ \Omega_{n'} \quad \forall~ n'\ne n .
\end{array}\right.
\end{eqnarray}
The basis function $\phi_{\ell m,\lmax}^{n}$ is constructed as \eqref{trial-basis} such that it (a) is the solution of radial Schr\"{o}dinger equation at energy $z$ times spherical harmonic function $Y_{\ell m}$ in the $n$-th atomic sphere, (b) extends to the interstitial region continuously by Hankel functions $H_{\ell m}$ as the scattered wave, (c) and flows into other atomic spheres in a weakly continuous way while satisfying the Schr\"{o}dinger equation with energy $z$ at the same time.
Therefore, $\wlz$ is a subset of $\ulz$.
Conversely, $\ulz$ contains more functions than $\wlz$, say $\wlz\subsetneq \ulz$,
since the slope of $\phi_{\ell m,\lmax}^{n}$ may not be weakly continuous through the atomic spherical surfaces.

We also need the following test function space
\begin{eqnarray}
\label{test-space}
\vlz = \left\{u\in H^{\delta}(\mathbb{R}^3) ~\Big|~ u = \sum_{n=1}^N\sum_{\ell\leq\lmax} c_{n\ell m} \varphi_{\ell m,\lmax}^{n}(\bm{r};z) \right\} ,
\end{eqnarray}
where the basis function $\varphi_{\ell m,\lmax}^{n}$ is defined by
\begin{eqnarray}
\label{test-basis}
\varphi_{\ell m,\lmax}^{n}(\bm{r};z) 
= \left\{
\begin{array}{ll}
\displaystyle \phi_{\ell m,\lmax}^{n}(\bm{r};z) - \eta_{n'}(|\bm{r}_{n'}|) \sum_{\ell'>\lmax} \struc(z) J_{\ell'm'}(\bm{r}_{n'};z)  
\qquad  & {\rm in} ~ B_{n'}~(n'\ne n) ,
\\[1ex]
\phi_{\ell m,\lmax}^{n}(\bm{r};z) & {\rm otherwise}
\end{array}\right.
~~
\end{eqnarray}
with $B_n=\{{\bm r}~\big|~ R_{n}<|{\bm r}_{n}|<R_{n}+r_{n}^{\rm b} \}$ \label{Bn} being a ring (of width $r_{n}^{\rm b}$) around the $n$-th atomic sphere (see Figure \ref{fig-domain}),
and $\eta_n$ being a cut-off functions satisfying $\eta_n\in C^{\infty}[0,+\infty)$, $0\le\eta_n\le 1$ and
\begin{eqnarray}\label{cut-off-eq}
\eta_n(r)=\left\{
\begin{array}{ll}
1 & \quad r\le R_n
\\[1ex]
0 & \quad r\ge R_n+r_{n}^{\rm b}
\end{array}
\right.
\quad{\rm for}~n=1,\cdots,N .
\end{eqnarray}
The widths $r_n^{\rm b}$ should be small enough such that the regions $B_n~(n=1,\cdots,N)$ do not overlap.
The basis function $\varphi_{\ell m,\lmax}^{n}$ can be viewed as a slight mollification of $\phi_{\ell m,\lmax}^{n}$ (only on the regions $B_{n'}$ with $n'\neq n$) by the cut-off functions, such that it is continuous (not just weakly continuous) through the atomic spherical surfaces $\Gamma_{n'}$.
Note that $\varphi_{\ell m,\lmax}^{n}$ does not satisfy the Schr\"{o}dinger equation in $B_n$ any more.

We then define a map $\mollf: \ulz\rightarrow \vlz$ such that
\begin{eqnarray}\label{mollify}
\mollf \phi_{\ell m,\lmax}^{n}(\bm{r};z) = \varphi_{\ell m,\lmax}^{n}(\bm{r};z) \qquad  \forall~ 1\le n \le N, ~0\le \ell\le \lmax, ~|m|\le \ell.
\end{eqnarray}
It can be seen from the definition that the role of $\mollf$ is to mollify the functions in $\ulz$ from weakly continuous to continuous through the atomic spherical surfaces.
Note that the mapping $\mollf$ actually depends on the energy parameter $z$ and the angular-momentum truncation $L$.
For simplicity of presentations, we will suppress this dependence whenever it is clear from the context.

We then construct the following variational form of an eigenvalue problem: 
Find $\laml\in \mathbb{R}$ and $\ul\in \trial$ such that $(u_L, \mollf u_L)=1$ and
\begin{eqnarray}
\label{Galerkin-form}
a_{L}(\ul,v) = \laml (\ul, \mollf v) \qquad \forall~ v\in \trial ,
\end{eqnarray}
where the bilinear form $a_{L}(\cdot,\cdot):~\ulz \times \ulz \rightarrow \mathbb{C}$ is defined by
\begin{eqnarray}\nonumber
\label{pg}
a_{L}(u,v) := a_{\delta}(u,\mollf v).
\end{eqnarray}
We mention that \eqref{Galerkin-form} can also be viewed as a Petrov-Galerkin scheme \cite{demkowicz10}:
Find $\laml\in \mathbb{R}$ and $\ul\in \trial$ such that $(u_L, \mollf u_L)=1$ and
\begin{eqnarray}\nonumber
a_{\delta}(\ul,v) = \laml (\ul, v) \qquad \forall ~v\in \test ,
\end{eqnarray}
where the trial function space $\ulz$ and the test function space $\vlz$ are (slightly) different.

Note that \eqref{Galerkin-form} is a nonlinear eigenvalue problem since both the bilinear form $a_L(\cdot,\cdot)$ and the trial function space $\trial$ depend on the eigenvalue $E_L$ under consideration.
Therefore the solution may be very difficult to find from a practical point of view.
However, we clarify that \eqref{Galerkin-form} is introduced only for theoretical analysis rather than practical calculations.

We will first show some basic properties of the bilinear form $a_{L}(\cdot,\cdot)$ in the following lemma, which implies that $a_L$ gives an inner product on $\ulz$, and is coercive and continuous with respect to the broken Sobolev norm $\|\cdot\|_{\delta}$.
The proof is given in \ref{appen-bilinear}.

\begin{lemma}
\label{lem-equ}
{\rm (a)} $a_{L}(\cdot,\cdot)$ is symmetric, i.e.,
\begin{eqnarray}\label{norm-sym}
	a_{L}(u,v)=a_{L}(v,u)^* \qquad \forall~ u,v\in \ulz.
\end{eqnarray}
{\rm (b)} If $L$ is sufficiently large, then there exist positive constants $C_1,~C_2,~C_3$ and $C_4$ that do not depend on $\lmax$, such that for any $u \in \ulz$,
\begin{align}
\label{norm-equ}
& {\rm (b1)} \qquad C_1 \|u\|^2_{\delta} \le a_{L}(u,u) \le C_2 \|u\|^2_{\delta}, 
\\[1ex] 
\label{M-eq}
& {\rm (b2)} \qquad \|\mollf u\|_{\ltwo} \le C_3 \|u\|_{\ltwo}
\qquad{\rm and}\qquad
\|\mollf u\|_{\delta} \le C_4 \|u\|_{\delta} .
\end{align}
\end{lemma}

Then we can derive the following theorem, which shows that the solutions of the KKR secular equation are equivalent to those of the variational form \eqref{Galerkin-form}.
The proof is given in \ref{append-equivalence}.

\begin{theorem}
\label{equivalence-them}
If the asymptotic assumption \eqref{assumption} is satisfied, then the following two statements are equivalent:
(a) $(E_\lmax,\ul)$ is an eigenpair of \eqref{Galerkin-form};
(b) $E_{L}$ is a solution of \eqref{kkr} and $\ul$ is given by \eqref{kkr-wavefun}.
\end{theorem}

With this key observation, the error estimates for the KKR equation can now be obtained from the analysis of the variational problem.
From this theorem, we see that although the choice of cutoff $\eta_n$ affects the construction of the test function space $\vlz$ as well as the operator $\mollf$, but the solutions of \eqref{Galerkin-form} do not depend on $\eta_n$.

\subsection{Convergence and {\em a prior} error estimates}
\label{subsec:nual-anal-galerkin}

Let $(E_L,u_L)$ solves \eqref{Galerkin-form} with the angular-momentum truncation $L\in\N$.
We will provide an error estimate for the approximate error with respect to $L$. 
% which can justify the convergence rate of KKR methods due to the equivalence shown in Theorem \ref{equivalence-them}.

Due to the coercivity of $a_L(\cdot, \cdot)$ stated in Lemma \ref{lem-equ}, the approximate eigenvalues $E_{\lmax}$ have a lower bound. 
% Furthermore, we observe from the definition \eqref{trial-space} that the sequence $\{E_L\}_{L\in\Z_+}$ is non-increasing as $L$ increases.
Note that we search the eigenvalue approximations in a finite energy window that contains the eigenvalues to be calculated, we will have limiting points in this region.
We will show that any limiting point is an eigenvalue of the original problem \eqref{model-weak-eq} and possesses some spectral convergence rate.

Let $E_{\infty}\in\R$ be any accumulation point, that is,
% Therefore, the approximate eigenvalues 
there exists a subsequence $\{E_{\lmax_j}\}_{j\in\Z_+}$ converge to $E_{\infty}$
\begin{eqnarray}
\label{eigenvalue-convergence}
\lim\limits_{j\rightarrow \infty} \big|E_{\lmax_j}-E_{\infty}\big|=0.
\end{eqnarray}
Then it follows from the coercivity in Lemma \ref{lem-equ} that 
\begin{eqnarray}
\label{eigenfunction-bound}
E_{L_j} = a_{L_j} (u_{L_j}, u_{L_j}) \ge C \|u_{L_j}\|^{2}_{\delta} ,
\end{eqnarray}
which indicates that $u_{L_j}$ is bounded in $H^{\delta}(\mathbb{R}^3)$. 
Consequently, there exists a weakly convergent subsequence $\{u_{L_{j_k}}\}_{k\in\Z_+}$, such that
\begin{eqnarray}
\label{weak-convergence}
u_{L_{j_k}} \rightharpoonup u_{\infty} \quad\text{in} ~ H^{\delta}(\mathbb{R}^3) \quad \text{as}~ k\rightarrow \infty.
\end{eqnarray} 
For simplicity of the presentation, we will still denote the weakly converged subsequence $u_{L_{j_k}}$ by $u_{\lmax}$ in the following.

The next theorem is the main result of this paper, which shows that the limiting pair $(E_{\infty},u_{\infty})$ is an eigenpair of the eigenvalue problem \eqref{model-weak-eq}, 
and the KKR approximations (i.e. the solutions of \eqref{Galerkin-form} or \eqref{kkr}) converge to the eigenpair super-algebraically fast as $\lmax$ increases.
The proof of this theorem is given in \ref{append-convergence}.

\begin{theorem}
\label{them-convergence}
Any limiting pair $(E_{\infty}, u_{\infty})$ is an eigenpair of \eqref{model-weak-eq}, that is
\begin{eqnarray}
\label{convergence}
a(u_{\infty},v) = E_{\infty} (u_{\infty}, v) \qquad \forall ~v\in H^{1}(\mathbb{R}^3) .
\end{eqnarray}
Moreover, if the asymptotic assumption \eqref{assumption} is satisfied and $\lmax$ is sufficiently large, then there exists a positive constant $C_{s}$ depending only on $u_{\infty}$ and $s$ such that
\begin{eqnarray}
\label{convergent-rate-eq}
|E_{\lmax}-E_{\infty}| + \|u_{\lmax}-u_{\infty}\|_{\delta} \le C_{s} \lmax^{1-s}
\qquad \forall~s>1.
\end{eqnarray}
\end{theorem}

\begin{remark}
Note that the estimate is not restricted to the case of single eigenvalue, but for general multiple eigenvalues.
Moreover, the result can also be generalized to problems with non-spherical potentials, as the spherically symmetric condition is not essential in the proof.
\end{remark}

\section{Numerical experiments}
\label{sec:numerics}
\setcounter{equation}{0}

In this section, we will present some numerical results to support our theory.
All the experiments are performed on a PC via Matlab codes.
To test the convergence with respect to the angular-momentum truncation $L$, the result obtained by using a sufficiently large $L$ is taken to be the reference (exact solution).

\vskip 0.3cm

\noindent
{\bf Example 1} (spherical potentials).
Consider a system with four atoms located at 
$\bm{R}_1=(0.6, 0.6, 1)$, $\bm{R}_2=(0.6, 1.2, 1)$, $\bm{R}_3=(1.2, 0.6, 1)$ and $\bm{R}_4=(1.2, 1.2, 1)$, respectively.
The potential takes the form of \eqref{mt-potential} with the same radial components
\begin{eqnarray}\nonumber
v_1(r) = v_2(r) = v_3(r) = v_4(r) = \left\{
\begin{array}{cc}
\displaystyle -\frac{10}{r} + \frac{10}{R} & |r|\le R \\[2ex]
\displaystyle 0 & \rm{otherwise}
\end{array}\right. 
\end{eqnarray}
with $R=0.2$.
We refer to Figure \ref{fig-coul-potential} 
for a schematic plot of the atomic configuration (four atomic spheres centred at $\pmb{R}_n~(n=1,\cdots,4)$ and the interstitial region $\Omega_{\rm I}$).

We solve the KKR equation by the ``root-tracing" methods.
More precisely, we scan some energy window to find the root of the determinant $S(E_L)$ with some given $L$.
We show in Figure \ref{fig-determinant-coul} the determinant as a function of energy parameter $z$ in the region of interest, with $L=3$.
We then present the decay of errors with respect to the angular-momentum truncation $L$ for the first eigenvalue in Figure \ref{fig-err-coul}. 
We observe that the error decays exponentially fast as $L$ increases, which is consistent with our theoretical result.

\begin{figure}[htbp]
\centering
\subfloat[]{
	\includegraphics[width=0.25\textwidth]{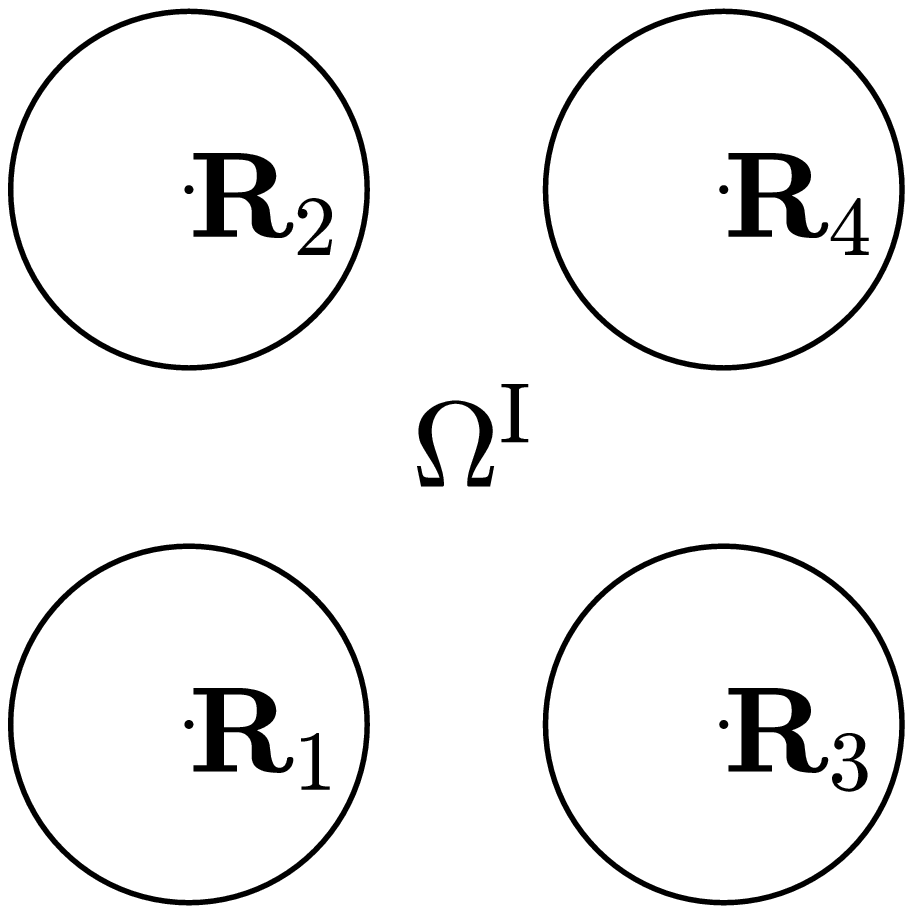}
	\label{fig-coul-potential}
}
\subfloat[]{
	\includegraphics[width=0.35\textwidth]{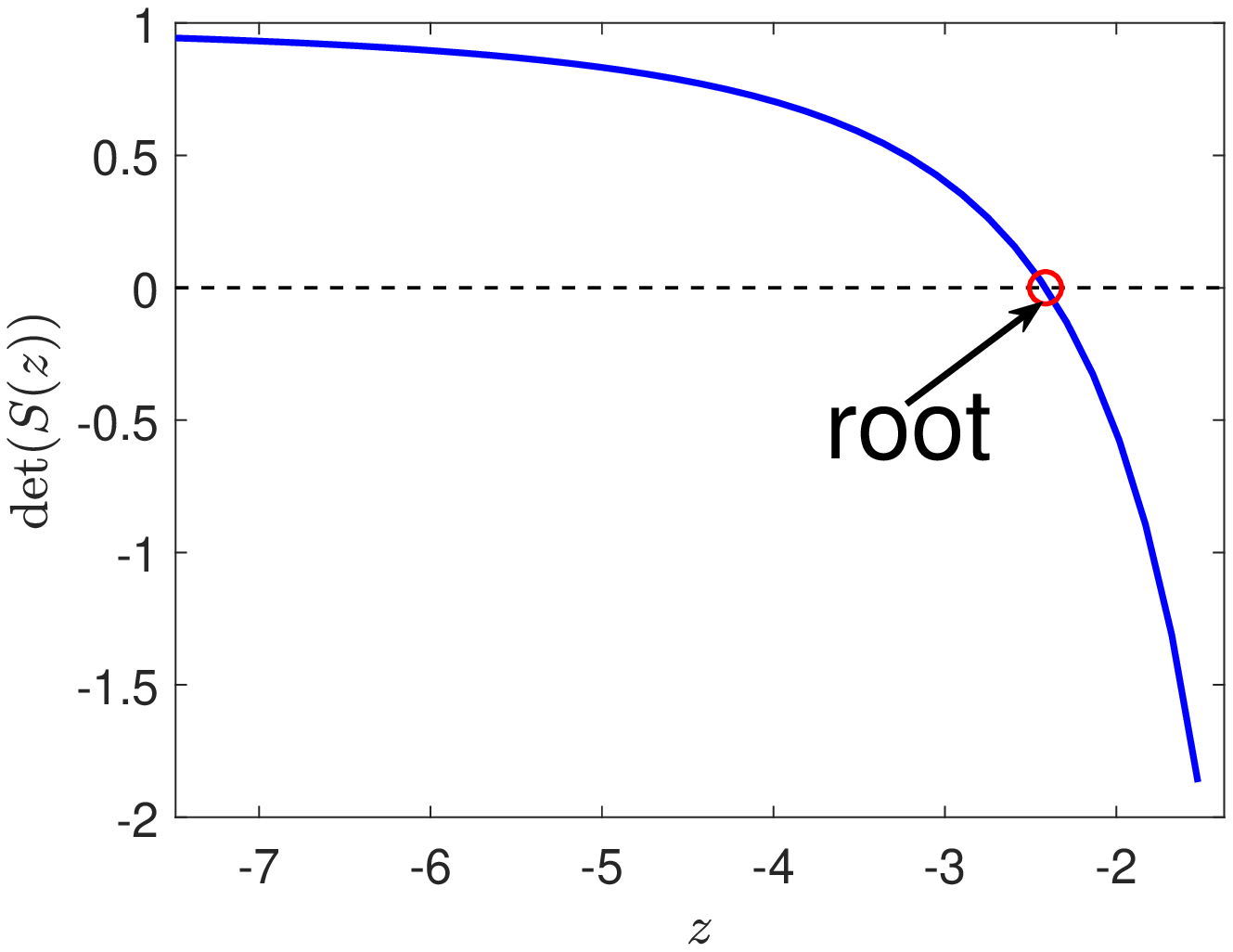}
	\label{fig-determinant-coul}
}
\subfloat[]{
	\includegraphics[width=0.35\textwidth]{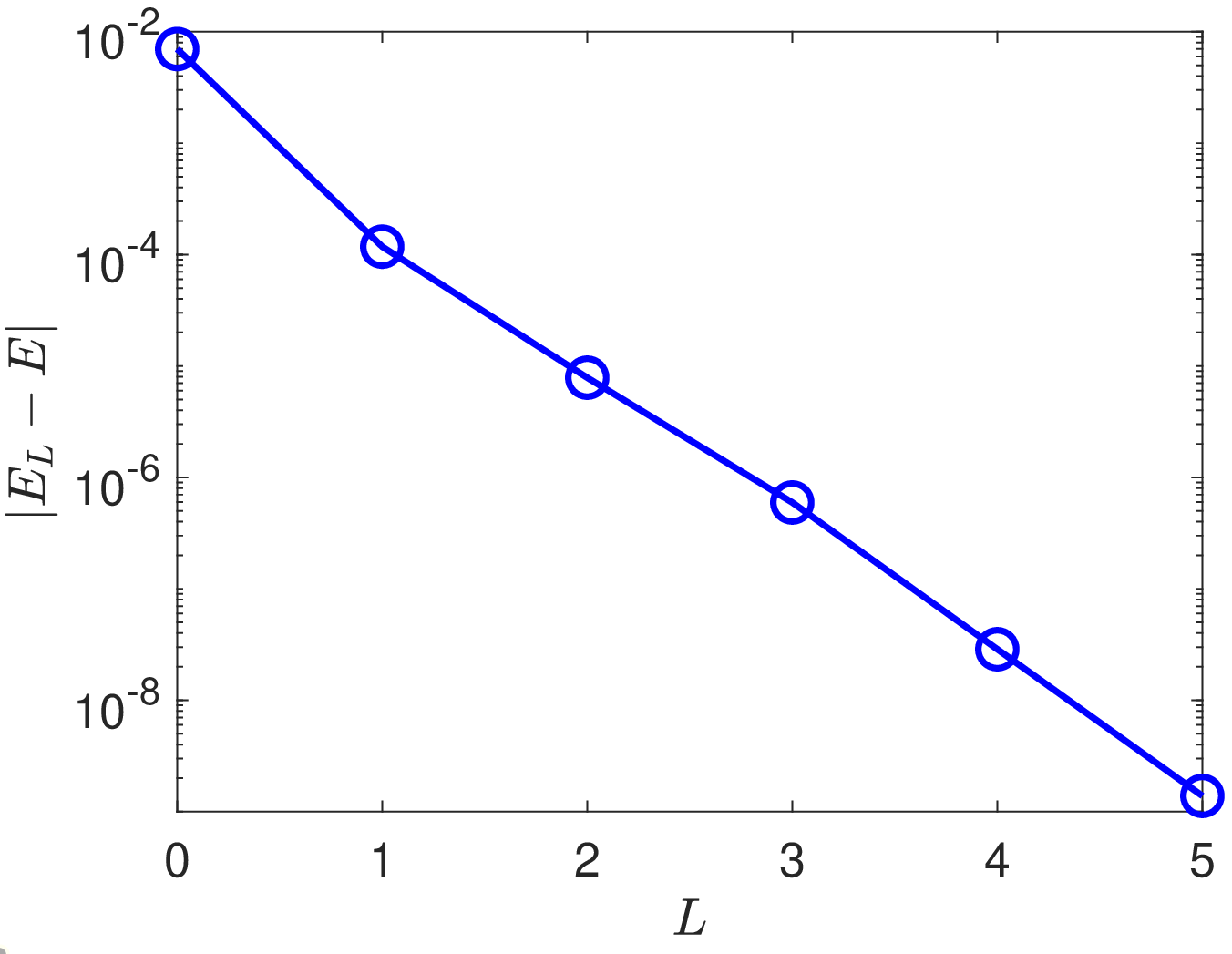}
	\label{fig-err-coul}
}	
\caption{Example 1. (a): Domain decomposition. (b): Solving the KKR equation by ``root-tracing". (c): Decay of the numerical errors with respect to angular-momentum truncation.}
\end{figure}

\vskip 0.3cm

\noindent
{\bf Example 2} (non-spherical potentials).
Consider a system with two atoms located at $\bm{R}_1=(0.5, 1, 1)$ and $\bm{R}_2=(1.5, 1, 1)$.
The radii for the atomic spheres are $R_1=0.35$ and $R_2=0.3$, respectively (see Figure \ref{fig-Gaussian-potential}).
We put non-spherical potentials in both atomic spheres, which has non-vanishing angular components for $\ell=1$.
In particular, the potential is given by
\begin{eqnarray}
\label{total-potential}
V(\bm{r}) = \left\{
\begin{array}{ll}
\displaystyle v^{1}_{00}(r_1) Y_{00} + \sum^{1}_{m=-1} v^{1}_{1m}(r_1) Y_{1,m}(\hat{\bm{r}}_1) & \quad\bm{r}\in \Omega_1, \\[1ex]
\displaystyle v^{2}_{00}(r_2) Y_{00} + \sum^{1}_{m=-1} v^{2}_{1m}(r_2) Y_{1,m}(\hat{\bm{r}}_2) & \quad\bm{r}\in \Omega_2, \\[2ex]
\displaystyle 0 & \quad\bm{r}\in \Omega^{\rm I} ,
\end{array}\right.
\end{eqnarray}
with the Gaussian-type radial parts
\begin{eqnarray}
\label{potential-gaussian}
v^{n}_{\ell m}(r) = c^{n}_{\ell m} r^{\ell} \left( e^{-(r/\sigma^{n}_{\ell m})^2} - e^{-(R_n/\sigma^{n}_{\ell m})^2} \right) \qquad {\rm for} ~ n=1,2, ~\ell=0,1, ~|m|\le \ell ,
\end{eqnarray}
where the coefficients are given in Table \ref{table-coef}.

\begin{table}  
\centering  
\begin{tabular}{ccccccccc}   
\toprule[1pt]  
\rowcolor[gray]{1} 
$(n,\ell,m)$ &$(1,0,0)$ &$(1,1,-1)$ &$(1,1,0)$  &$(1,1,1)$ &$(2,0,0)$ &$(2,1,-1)$ &$(2,1,0)$  &$(2,1,1)$ \\  
\midrule  
$c^{n}_{\ell m}$   &-700  &-30  &-20  &-15  &-720 &-70 &-30 &-20 \\  
$\sigma^{n}_{\ell m}$  &0.1   &0.05  &0.04  &0.045  &0.1 &0.04 &0.03 &0.04 \\  
\bottomrule[1pt]  
\end{tabular}  
\caption{The coefficients for the potential in \eqref{total-potential} and \eqref{potential-gaussian}.}
\label{table-coef}
\end{table} 

We use the ``root-tracing" method to solve the KKR equation, and show the behavior of determinant (for $L=3$) in Figure \ref{fig-determinant-nonsym}. 
The decay of numerical errors with respect to the angular-momentum truncation $L$ for the first eigenvalue is presented in Figure \ref{fig-err-nonsym}, from which we observe a spectral convergence rate.

\begin{figure}[htbp]
\centering
\subfloat[]{
	\includegraphics[width=0.25\textwidth]{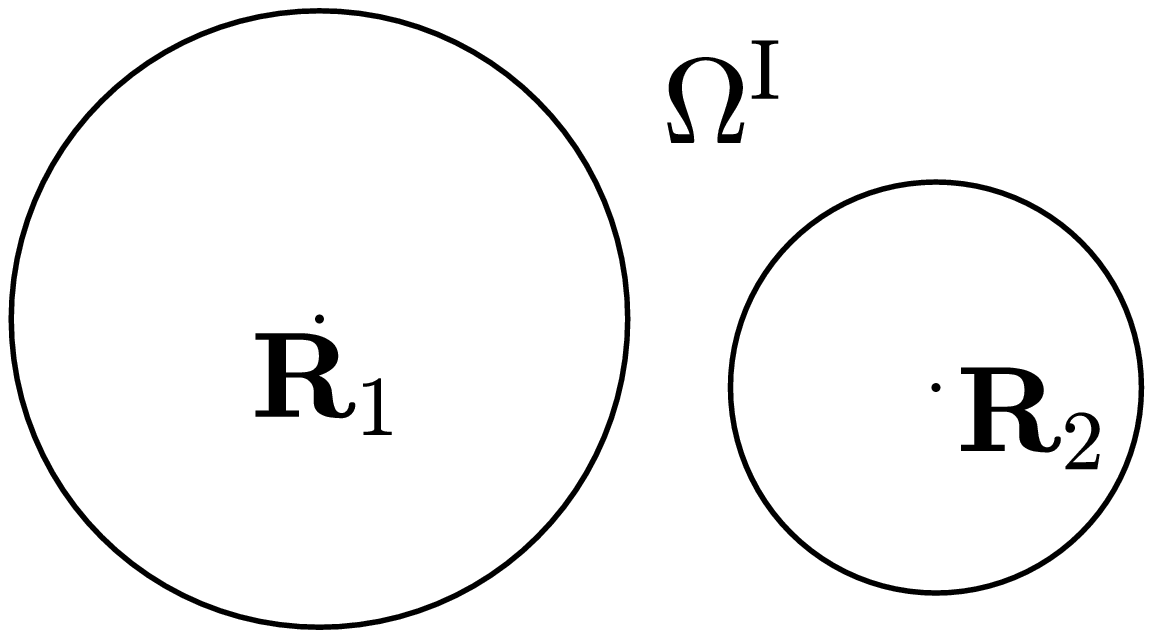}
	\label{fig-Gaussian-potential}
}
\subfloat[]{
	\includegraphics[width=0.35\textwidth]{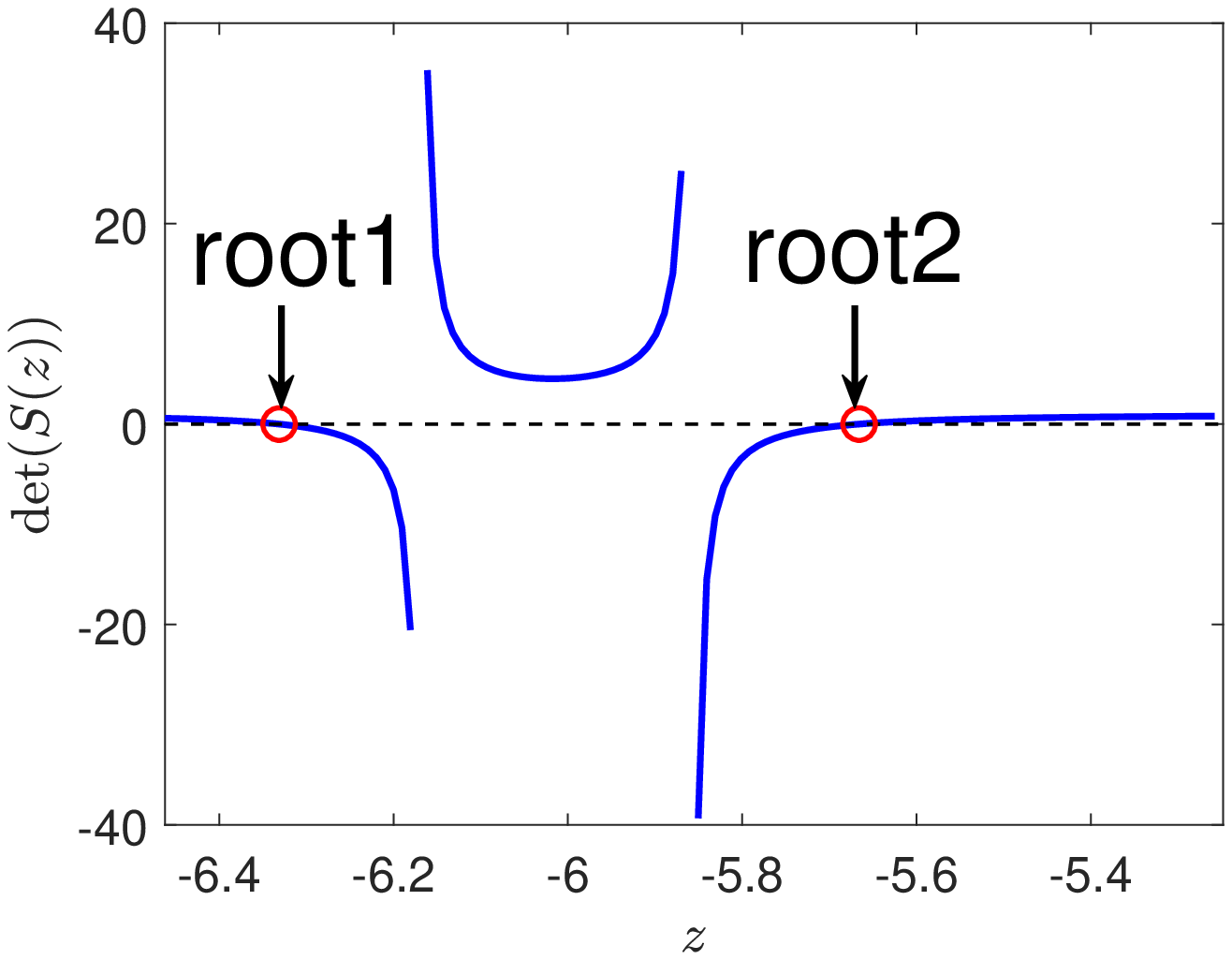}
	\label{fig-determinant-nonsym}
}
\subfloat[]{
	\includegraphics[width=0.35\textwidth]{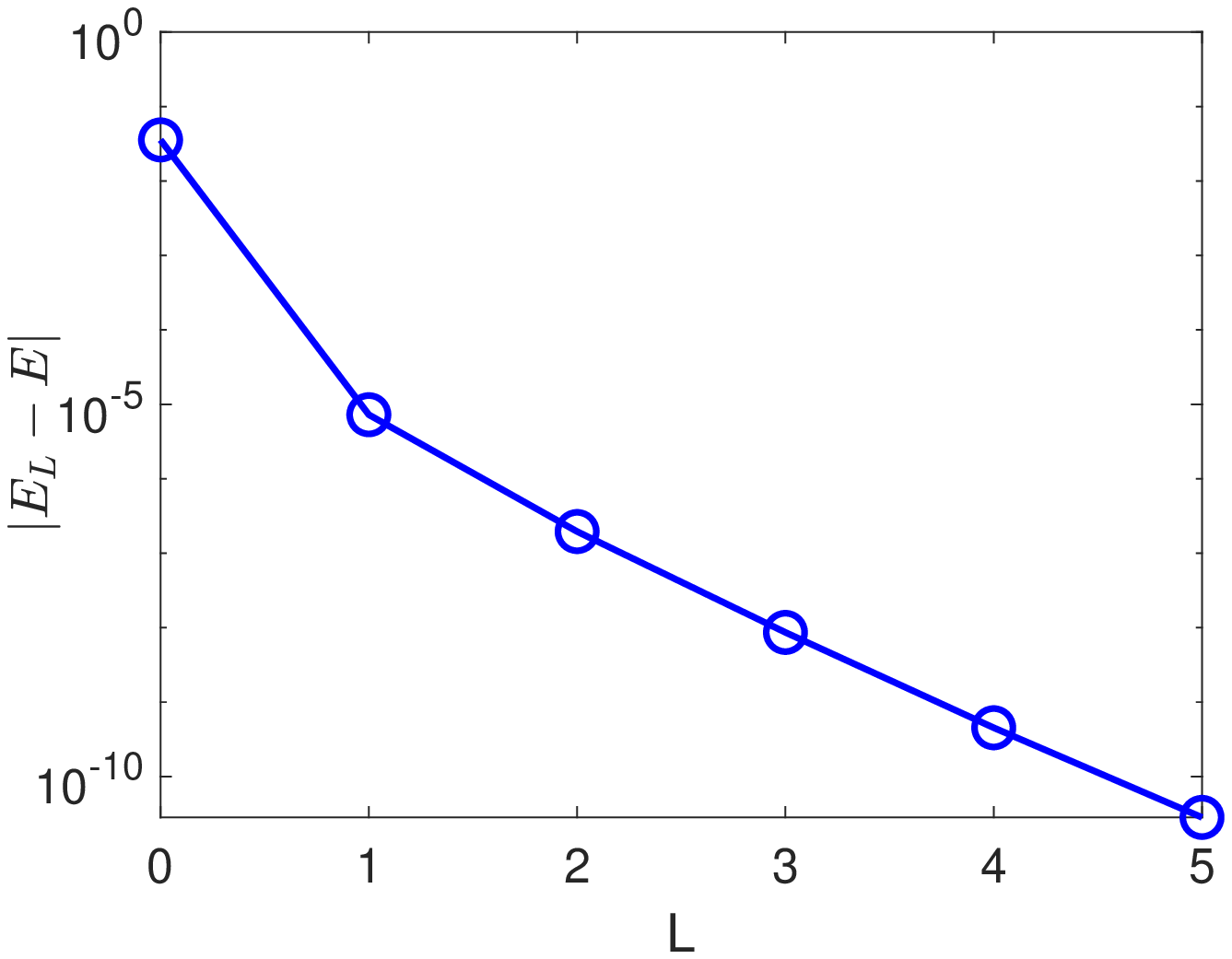}
	\label{fig-err-nonsym}
}		
\caption{Example 2. (a): Domain decomposition. (b): Solving the KKR equation by ``root-tracing". (c): Decay of the numerical errors with respect to angular-momentum truncation.}
\end{figure}

% Section Conclusions

\section{Concluding remarks}
\label{sec:conclusions}
\setcounter{equation}{0}

In this paper, we investigate the KKR secular equations and justify the spectral convergence rate with respect to the truncation of the angular-momentum sums.
Our theory works for almost all MST methods (not only the KKR equation but also other forms) due to their mathematical equivalence in the angular-momentum representations.
We further point out that since the MST methods are essentially Green's function methods, the angular-momentum truncation on the wave function expansion can also be viewed as a truncation of the free electron Green's function expansion.
In particular, it approximates \eqref{free-green-angular-momentum} by
\begin{eqnarray*}
G_{0}(\bm{r}, \bm{r}';z) \approx \sum_{\ell\leq L} \ylm g_{\ell}(r,r';z)  Y_{\ell m}(\hat{\bm{r}}') .
\end{eqnarray*}
We will present an analysis based on the approximation of Green's function in our future work.

The formalism and theory in this paper are made within the muffin-tin approximation of the potentials, which confine the potentials in non-overlapping atomic spheres (scattering cells).
Although the muffin-tin approximation has been appreciated in electronic structure calculations of metals and many of their alloys, it is necessary to generalize the theory to non-muffin-tin cases (space-filling potential cells) for systems with low rotational symmetry, for example, covalently bonded materials, surfaces, various kinds of impurities, and so on \cite{gonis00}.
The extension of our numerical analysis to MST methods with general space-filling potential cells is highly non-trivial, which will be studied in another work.

% APPENDIX
\appendix
\renewcommand\thesection{\appendixname~\Alph{section}}

% sec:proofs

\section{Proofs}
\label{append-proofs}

\subsection{Proof of Lemma \ref{lem-weak-continuous}}
\label{append-weak-continuity}
\renewcommand{\theequation}{A.1.\arabic{equation}}
\setcounter{equation}{0}  

\begin{proof}
Let $E_L$ be the solution of the KKR secular equation \eqref{kkr}, and $\ul$ be given by \eqref{kkr-wavefun} with the coefficients $a^{n}_{\ell m}$ and $b^{n}_{\ell m}$ satisfying \eqref{eq:alm} and \eqref{multi-incoming-scattering}.
To show that $u_L\in\mathcal{W}_L(E_L)$, it is sufficient to demonstrate that $u_L$ and its first derivative are weakly continuous through $\gamn$.
We first evaluate the jump of $u_L$ on $\Gamma_n$ by using \eqref{expansion-h-j} and \eqref{coef-in-sc}
\begin{align*}
& \hskip -0.2cm
[u_L]_{\Gamma_n}= ~\ul^{+}|_{\gamn} - \ul^{-}|_{\gamn} \\[1ex]
& = \sumlL a^{n}_{\ell m} j_{\ell}(\sqrt{E_{\lmax}} R_n) Y_{\ell m}(\hat{{\bm r}}_{n}) 
- \sumnn \sum_{\ell'\le L} b^{n'}_{\ell' m'} H_{\ell'm'}({\bm r}_{n'};E_{\lmax}) \bigg|_{\gamn}  \\[1ex]
& = \sumlL a^{n}_{\ell m} j_{\ell}(\sqrt{E_{\lmax}} R_n) Y_{\ell m}(\hat{{\bm r}}_{n}) 
- \sumlm \left( \sumnn \sum_{\ell'\le L} g_{\ell m,\ell'm'}^{nn'}(E_{\lmax}) b^{n'}_{\ell'm'} \right)
j_{\ell}(\sqrt{E_{\lmax}} R_n) Y_{\ell m}(\hat{{\bm r}}_{n})  
\\[1ex]
& = -\sum_{\ell>\lmax} a^{n}_{\ell m} j_{\ell}(\sqrt{E_{\lmax}}R_n) Y_{\ell m}(\hat{{\bm r}}_{n}).
\end{align*}
We then immediately obtain $\int_{\Gamma_n} [u_L]_{\Gamma_n}Y_{\ell m} =0~(\forall~\ell\leq L)$ by the orthogonality of the spherical harmonic functions. 
Similarly, we have that for the first derivative, $\int_{\Gamma_n} [\partial u_L/\partial r_n]_{\Gamma_n} Y_{\ell m} =0~(\forall~\ell\leq L)$.
Therefore, $u_L \in \mathcal{W}_L(E_L)$.
	
Next, we will show that if $\mathcal{W}_L(E_L)\backslash \{0\} \ne\emptyset$, then $E_L$ solves \eqref{kkr} and any $0\neq u_{L}\in \mathcal{W}_{\lmax}(E_{\lmax})$ can be written as \eqref{kkr-wavefun}.
Since $u_L\in\mathcal{W}_{\lmax}(E_{\lmax})$, we can write it as
\begin{eqnarray}\nonumber
u_{L}(\bm r) = \left\{
\begin{array}{ll}
\displaystyle \sum_{\ell\le \lmax} A^{n}_{\ell m} \zeta^{n}_{\ell m}(\bm{r}_{n};E_{\lmax})  & {\rm in} ~ \omen~(n=1,\cdots, N), \\[3ex]
\displaystyle \sumn \sum_{\ell \le \lmax} B^{n'}_{\ell m} H_{\ell m}(\bm{r}_{n'};E_{\lmax}) & \text{in}~\omeout,
\end{array}
\right.
\end{eqnarray}
with $A^{n}_{\ell m}$ and $B^{n}_{\ell m}$ the coefficients and $\zeta^{n}_{\ell m}$ given by \eqref{radial-component}.
We then have
\begin{align*}
 %\hskip -0.2cm
[u_L]_{\Gamma_n} 
%= ~u_{L}^{+}|_{\gamn} - u_{L}^{-}|_{\gamn} \\[1ex]
&= \sumlL A^{n}_{\ell m} \zeta_{\ell m}^{n}(\bm{r}_{n};\laml) \bigg|_{\gamn} - \sumn \sumlL B^{n'}_{\ell m} H_{\ell m}(\bm{r}_{n'};\laml) \bigg|_{\gamn} 
\\[1ex]
& = \sumlL A^{n}_{\ell m} \left( j_{\ell}(\sqrt{\laml}R_n) + t^{n}_{\ell}(\laml) h_{\ell}(\sqrt{\laml}R_n) \right)  Y_{\ell m}(\hat{{\bm r}}_{n}) 
- \sumlL B^{n}_{\ell m} h_{\ell}(\sqrt{\laml}R_n)  Y_{\ell m}(\hat{{\bm r}}_{n}) \\[1ex]
& \quad - \sum_{\ell m} \left( \sumnn \sum_{\ell'\le L} g_{\ell m,\ell'm'}^{nn'}(\laml) B^{n'}_{\ell'm'} \right) 
j_{\ell}(\sqrt{\laml}R_n)  Y_{\ell m}(\hat{{\bm r}}_{n}),
\end{align*}
which together with the weak continuity of $\ul$ (since $\ul\in \mathcal{W}_{\lmax}(E_{\lmax})$) implies
\begin{align}\nonumber
\left( A^{n}_{\ell m} - \sumnn \sum_{\ell'\le L} g_{\ell m,\ell'm'}^{nn'}(\laml) B^{n'}_{\ell'm'}  \right) &j_{\ell}(\sqrt{\laml}R_n) 
+ \Big( t^{n}_{\ell}(\laml) A^{n}_{\ell m} - B^{n}_{\ell m} \Big)  h_{\ell}(\sqrt{\laml}R_n) = 0, \\ \label{proof-1-a}
&1\le n\le N, ~0\le \ell\le \lmax, ~|m|\le \ell.
\end{align}
Using a similar calculation and the weak continuity of $\partial \ul/\partial r_n$, we have
\begin{align}\nonumber
\left( A^{n}_{\ell m} - \sumnn \sum_{\ell'\le L} g_{\ell m,\ell'm'}^{nn'}(\laml) B^{n'}_{\ell'm'} \right) &j'_{\ell}(\sqrt{\laml}R_n) + \Big( t^{n}_{\ell}(\laml) A^{n}_{\ell m} - B^{n}_{\ell m} \Big) h'_{\ell}(\sqrt{\laml}R_n) = 0, \\ \label{proof-1-b}
&1\le n\le N, ~0\le \ell\le \lmax, ~|m|\le \ell.
\end{align}
Combining \eqref{proof-1-a}, \eqref{proof-1-b}, and the following Wronskian relation
\begin{eqnarray*}
\begin{vmatrix}
j_{\ell}(x) & h_{\ell}(x) \\
j'_{\ell}(x) & h'_{\ell}(x)
\end{vmatrix}
= \frac{i}{x^2} \ne 0,
\end{eqnarray*}
we have
\begin{eqnarray}
\label{proof-1-d}
\left\{
\begin{array}{l}
\displaystyle A^{n}_{\ell m} =  \sumnn \sum_{\ell'\le L} g_{\ell m, \ell'm'}^{nn'}(\laml) B^{n'}_{\ell'm'} \\[3ex]
B^{n}_{\ell m} = t^{n}_{\ell}(\laml) A^{n}_{\ell m}
\end{array}\right.
\quad 1\le n\le N, ~0\le \ell\le \lmax, ~|m|\le \ell.
\end{eqnarray}
Note that \eqref{proof-1-d} and $u_L\neq 0$ indicate that $\{A^{n}_{\ell m}\}$ is a non-trivial solution of \eqref{eq:alm} and hence $E_{\lmax}$ solves the KKR secular equation \eqref{kkr}.
This completes the proof.
\end{proof}

\subsection{Proof of Lemma \ref{lem-equ}}
\label{appen-bilinear}
\renewcommand{\theequation}{A.2.\arabic{equation}}
\setcounter{equation}{0}  

\begin{proof} [Proof of Lemma \ref{lem-equ} (a)]
In order to show the symmetry of $a_{L}(\cdot,\cdot)$, it is only necessary to show that the following equality holds for any basis function $\phi^n_{\ell m,L}$ (defined in \eqref{trial-basis})
\begin{multline}
\label{basis-sym}
\qquad
a_{L}(\phi^{n_1}_{\ell_1 m_1,\lmax},\phi^{n_2}_{\ell_2 m_2,\lmax})
= a_{L}(\phi^{n_2}_{\ell_2 m_2,\lmax}, \phi^{n_1}_{\ell_1 m_1,\lmax})^*  \\
\forall~1\le n_1,n_2\le N, ~0\le \ell_1,\ell_2\le \lmax, ~|m_1|\le\ell_1, ~|m_2|\le \ell_2 . \qquad
\end{multline} 
According to the definition of $a_L(\cdot,\cdot)$, \eqref{basis-sym} can be written in terms of $\varphi^n_{\ell m,L}$ (defined in \eqref{test-basis}) as
\begin{eqnarray}
\label{basis-sym-uv}
a_{\delta}(\phi^{n_1}_{\ell_1 m_1,\lmax},\varphi^{n_2}_{\ell_2 m_2,\lmax})
= a_{\delta}(\phi^{n_2}_{\ell_2 m_2,\lmax}, \varphi^{n_1}_{\ell_1 m_1,\lmax})^*.
\end{eqnarray}
Furthermore, since $a_{\delta}(\cdot,\cdot)$ is symmetric and the basis functions $\phi^n_{\ell m}$ and $\varphi^n_{\ell m}$ only differ in the regions $B_{n'}~(n'\neq n)$, it is only necessary for us to consider the above integrals (in \eqref{basis-sym-uv}) in $B_{n}$ for $1\le n\le N$.
In particular, we need to show that for any $1\le n\le N$,
\begin{subequations}
\begin{align}
\label{L2-sym}
& \int_{B_n} \phi^{n_1}_{\ell_1 m_1,\lmax} \varphi^{n_2 *}_{\ell_2 m_2,\lmax} = \int_{B_n} \phi^{n_2 *}_{\ell_2 m_2,\lmax} \varphi^{n_1}_{\ell_1 m_1,\lmax} 
\qquad\qquad {\rm and}
\\[1ex]
\label{H1-sym}
& \int_{B_n} \nabla\phi^{n_1}_{\ell_1 m_1,\lmax} \cdot \nabla \varphi^{n_2 *}_{\ell_2 m_2,\lmax} = \int_{B_n} \nabla \phi^{n_2 *}_{\ell_2 m_2,\lmax} \cdot \nabla \varphi^{n_1}_{\ell_1 m_1,\lmax} .
\end{align}
\end{subequations}
We will consider \eqref{L2-sym} and \eqref{H1-sym} for three cases separately in the following: (i) $n=n_1=n_2$, (ii) $n \ne n_1$ and $n \ne n_2$, and (iii) $n = n_1 \ne n_2$ (the same for $n = n_2 \ne n_1$).

The first case with $n=n_1=n_2$ is clear due to the fact that $\varphi^{n}_{\ell m,\lmax}=\phi^{n}_{\ell m,\lmax}$ in $B_n$.

For the second case with $n\ne n_1$ and $n\ne n_2$, \eqref{L2-sym} can be easily verified by the definitions \eqref{trial-basis}, \eqref{test-basis} and a direct calculation.
Then by using the definition of the cutoff function $\eta_{n}$ and the fact that $\phi^n_{\ell m}$ satisfies \eqref{bvp-schrodinger} on $B_{n'}~(\forall~n')$, we have
\begin{align}
\label{basis-sym-h1}
\nonumber
&\int_{B_n} \left( \nabla \phi^{n_1}_{\ell_1 m_1,\lmax} \cdot \nabla \varphi^{n_2 *}_{\ell_2 m_2,\lmax}  - \nabla \phi^{n_2 *}_{\ell_2 m_2,\lmax} \cdot \nabla\varphi^{n_1}_{\ell_1 m_1,\lmax} \right) 
\\[1ex]\nonumber
= &\int_{B_n}  \left( -\Delta \phi^{n_1}_{\ell_1 m_1,\lmax} \varphi^{n_2 *}_{\ell_2 m_2,\lmax} 
+ \varphi^{n_1}_{\ell_1 m_1,\lmax} \Delta\phi ^{n_2 *}_{\ell_2 m_2,\lmax}  \right) 
+ \int_{\partial B_n}  \left( \frac{\partial \phi^{n_1}_{\ell_1 m_1,\lmax} }{\partial r_n} \varphi^{n_2 *}_{\ell_2 m_2,\lmax} - \frac{\partial \phi^{n_2 *}_{\ell_2 m_2,\lmax} }{\partial r_n} \varphi^{n_1}_{\ell_1 m_1,\lmax} \right) 
\\[1ex]\nonumber
= & z\int_{B_n} \left( \phi^{n_1}_{\ell_1 m_1,\lmax} \varphi^{n_2 *}_{\ell_2 m_2,\lmax} - \phi^{n_2 *}_{\ell_2 m_2,\lmax} \varphi^{n_1}_{\ell_1 m_1,\lmax} \right)
+ \int_{r_n = R_n+r_n^{\rm b}} \left( \frac{\partial \phi^{n_1}_{\ell_1 m_1,\lmax} }{\partial r_n} \phi^{n_2 *}_{\ell_2 m_2,\lmax} - \frac{\partial \phi^{n_2 *}_{\ell_2 m_2,\lmax} }{\partial r_n} \phi^{n_1}_{\ell_1 m_1,\lmax} \right)
\\[1ex] 
& + \int_{\gamn} \left( \frac{\partial \phi^{n_1}_{\ell_1 m_1,\lmax} }{\partial r_n} \varphi^{n_2 *}_{\ell_2 m_2,\lmax} - \frac{\partial \phi^{n_2 *}_{\ell_2 m_2,\lmax} }{\partial r_n} \varphi^{n_1}_{\ell_1 m_1,\lmax} \right) .
\end{align}
Note that a simple calculation can indicate that each of the three terms in \eqref{basis-sym-h1} vanishes.
Therefore, \eqref{H1-sym} is also true for the second case.

For the last case with $n=n_1\ne n_2$, we can use similar arguments and the fact that $\phi^{n}_{\ell m,\lmax}=H_{\ell m}$ in $B_{n}$ to show  \eqref{L2-sym} and \eqref{H1-sym}.
\end{proof}

\begin{proof}[Proof of Lemma \ref{lem-equ} (b)]
It is only necessary to show that there exists a constant $0<\gamma<1$, independent of $L$, such that
\begin{eqnarray}
\label{norm-equ-control}
|a_{\delta}(u,\mollf u-u)|\le \gamma a_{\delta}(u,u)
\qquad \forall~ u\in \ulz .
\end{eqnarray}
We will consider the right-hand side and left-hand side of \eqref{norm-equ-control} respectively in the following.
	
Since $a_{\delta}(\cdot,\cdot)$ is coercive in any region, the right-hand side of \eqref{norm-equ-control} can be estimated by 
\begin{eqnarray}
\label{proof-1-e}
a_{\delta}(u,u) \ge\int_{\omeout} |\nabla u|^2 \ge \sum_{n=1}^{N} \int_{\widetilde{B}_n}  |\nabla u|^2,
\end{eqnarray}
where $\widetilde{B}_{n}=\big\{ {\bm r}\in\omeout ~\big|~ R_{n}<|{\bm r}_{n}|<R_{n}+2r_n^{\rm b} \big\}$.
In each $\widetilde{B}_{n}$, $u\in\ulz$ can be represented in the spherical coordinates with respect to ${\bm R}_n$ as
\begin{eqnarray}
\label{u-expand}
u({\bm r})=\sum_{\ell m} a^{n}_{\ell m} J_{\ell m}({\bm r}_n;z) + \sum_{\ell\le L} b^{n}_{\ell m} H_{\ell m}({\bm r}_n;z).
\end{eqnarray}
Using \eqref{u-expand} and $\Delta=\frac{1}{r^2} \frac{\partial}{\partial r}(r^2 \frac{\partial}{\partial r}) + \frac{1}{r^2} \Delta_{S^2}$ with $\Delta_{S^2}$ the Laplace-Beltrami operator on the unit sphere $S^2$, we have
\begin{align}
\label{sym-bilinear}
\nonumber
\int_{\widetilde{B}_n} |\nabla u|^2
& = \int_{R_n}^{R_n+2 r_n^{\rm b}} r^2\int_{S^2} \left| \frac{\partial u}{\partial r} \right|^2 
- \int_{R_n}^{R_n+2r_n^{\rm b}} \int_{S^2} u^* \Delta_{S^2}u 
\\[1ex]
& \ge \displaystyle \sum_{\ell>\lmax} |a_{\ell m}^{n}|^2 \int_{R_n}^{R_n+2 r_n^{\rm b}} \Big( r^2|z| |j'_{\ell}(\sqrt{z}r)|^2  
+ \ell(\ell+1) \big|j_{\ell}(\sqrt{z}r)\big|^2 \Big),
\end{align}
where all the $\ell\le L$ components are positive and dropped for the last inequality.

For the left-hand side of \eqref{norm-equ-control}, we have from the definition of $\mollf$ that
\begin{eqnarray}
\label{proof-1-g}
	a_{\delta}(u, \mollf u-u) = \sum_{n=1}^{N} \int_{B_n} \nabla u \cdot \nabla (\mollf u-u)^*.
\end{eqnarray}
By writing $\mollf u$ as
\begin{eqnarray*}
	(\mollf u)(\bm{r}) =\sum_{\ell\le L} \Big( a^{n}_{\ell m} J_{\ell m}({\bm r}_n;z) +  b^{n}_{\ell m} H_{\ell m}({\bm r}_n;z) \Big)
	+ \Big(1 - \eta_{n}(|{\bm r}_n|)\Big) \sum_{\ell>L} a^{n}_{\ell m} J_{\ell m}({\bm r}_n;z) ,
\end{eqnarray*}
we have that for each $B_n$, 
\begin{align}
\label{proof-1-f}
\nonumber
& \hskip -0.2cm
\int_{B_n} \nabla u \cdot \nabla (\mollf u-u)^* = \int_{R_n}^{R_n+r_n^{\rm b}} r^2\int_{S^2}  \frac{\partial(\mollf u-u)^*}{\partial r} \frac{\partial u}{\partial r} 
- \int_{R_n}^{R_n+r_n^{\rm b}} \int_{S^2} (\mollf u-u)^* \Delta_{S^2}u \\[1ex]\nonumber
= & \sum_{\ell>\lmax} |a^{n}_{\ell m}|^2  \int_{R_n}^{R_n+r_n^{\rm b}} \Big( -|z|\eta_{n}(r) r^2 |j'_{\ell}(\sqrt{z}r)|^2 
- \sqrt{z} \eta'_{n}(r) r^2 j'_{\ell}(\sqrt{z}r) j^*_{\ell}(\sqrt{z}r) 
\\[1ex]
&\qquad\qquad\qquad\qquad\quad 
- \ell(\ell+1) \int_{R_n}^{R_n+r_n^{\rm b}} \eta_{n}(r) \big|j_{\ell}(\sqrt{z}r)\big|^2  \Big) .
\end{align}
Then we have from \eqref{proof-1-f} and the Cauchy-Schwartz inequality that for any $\epsilon>0$
\begin{multline}
\label{sym-err}
\quad
\left|\int_{B_n} \nabla u \cdot \nabla (\mollf u-u)^*\right| 
\\
\leq \sum_{\ell>\lmax} |a^{n}_{\ell m}|^2  \int_{R_n}^{R_n+r_n^{\rm b}} \Big( (1+\epsilon)r^2|z| \big|j'_{\ell}(\sqrt{z}r)\big|^2 
+ \Big(\ell(\ell+1)+\frac{1}{\epsilon} \Big)  \big|j_{\ell}(\sqrt{z}r)\big|^2 \Big) .
\qquad
\end{multline}

Note that for $z<0$, $j_{\ell}(\sqrt{z}r)$ can be written by
\begin{eqnarray}
\label{analytic-bessel}
j_{\ell}(\sqrt{z}r) = i^{\ell} \frac{\sqrt{\pi}}{2^{\ell+1}} (kr)^{\ell} \sum_{m=0}^{\infty} \frac{k^{2m}}{m!\varGamma(m+\ell+3/2)}
\left( \frac{r}{2} \right)^{2m}
\end{eqnarray}
with $k=\sqrt{|z|}$ and $\varGamma$ denoting the Gamma function.
We can see that $|j_{\ell}(\sqrt{z}r)|$ and $|j'_{\ell}(\sqrt{z}r)|$ are increasing functions with respect to $r$ for any $\ell\in\N$.
% Let $\epsilon$ be small enough such that  $\lmax(\lmax+1)>1/\delta$.
By taking into accounts this property and \eqref{proof-1-e}, \eqref{sym-bilinear}, \eqref{proof-1-g} and \eqref{sym-err}, we can obtain \eqref{norm-equ-control} by appropriately choosing some $0<\epsilon<1$.
% only depending on $\|\eta_n\|_{C^\infty}$, $R_n$ and $r^{\rm b}_n$.
This completes the proof of Lemma \ref{lem-equ} (b1).

Lemma \ref{lem-equ} (b2) can be proved by almost the same arguments, and we skip the details for simplicity of the presentation. 
\end{proof}

\subsection{Proof of Theorem \ref{equivalence-them}}
\label{append-equivalence}
\renewcommand{\theequation}{A.3.\arabic{equation}}
\setcounter{equation}{0}  % number counter 

\begin{proof}
Let $E_{L}$ be the solution of \eqref{kkr} and $\ul$ be given by \eqref{kkr-wavefun}.
We have from Lemma \ref{lem-weak-continuous} that $\ul\in \mathcal{W}_L(E_L)$, and hence both $\ul$ and its first derivative are weakly continuous. 
Then by using the expression \eqref{kkr-wavefun} and the weak continuity, we can see by a direct calculation that $(E_{\lmax},\ul)$ satisfies the equation \eqref{Galerkin-form}. 
	
Conversely, let $(\laml,\ul)\in\R\times\mathcal{U}_{\lmax}(E_{\lmax})$ be an eigenpair of \eqref{Galerkin-form}. 
According to Lemma \ref{lem-weak-continuous}, it is only necessary for us to show that $\ul\in \mathcal{W}_L(E_L)$.
By comparing the definitions of $\mathcal{U}_{\lmax}(E_{\lmax})$ and $\mathcal{W}_L(E_L)$, we only need to prove that $\partial \ul/\partial r_{n}$ is weakly continuous through $\Gamma_n$ for $1\le n\le N$.
First, integrating the left-hand side of \eqref{Galerkin-form} by parts gives
\begin{eqnarray}
\label{surface-integ-sum}
\sum_{n=1}^{N} \int_{\Gamma_{n}} v \left[ \frac{\partial\ul}{\partial r_{n}} \right]_{\gamn}  \dd\hat{\bm{r}}_{n} =0 \qquad \forall ~v\in\test.
\end{eqnarray}
Then we note that the basis functions $\varphi_{\ell m,\lmax}^{n}$ of $\test$ on the spherical surface $\Gamma_{n'}~(n'\neq n)$ can be written by
\begin{align}
\label{basis-surface}
\nonumber
\varphi_{\ell m,\lmax}^{n}(\bm{r};E_{L}) \Big|_{\Gamma_{n'}} 
& = \displaystyle \sum_{\ell'\le L} h_{\ell' }(\sqrt{E_L}R_{n'}) \Big(  \delta_{nn'} \delta_{\ell\ell'} \delta_{mm'} + \frac{j_{\ell'}(\sqrt{\laml}R_{n'}) }{h_{\ell'}(\sqrt{\laml}R_{n'})} \struc(\laml)  \Big) Y_{\ell'm'}(\hat{{\bm r}}_{n'})
\\[1ex]
& = \sum_{\ell'\le L}  h_{\ell' }(\sqrt{E_L}R_{n'}) P_{n\ell m,n'\ell'm'}(E_L) Y_{\ell'm'}(\hat{{\bm r}}_{n'}) 
\end{align}
with $P_{n\ell m,n'\ell'm'}$ given by \eqref{assump-inv}.
By substituting \eqref{basis-surface} into \eqref{surface-integ-sum} with $v(\bm{r})=\varphi_{\ell m,\lmax}^{n}(\bm{r};E_{L})$, we can derive from the assumption \eqref{assumption} (for asymptotic problem) that
\begin{eqnarray}
\label{comp-surface}
\int_{\Gamma_{n}} \left[ \frac{\partial\ul}{\partial r_{n}} \right]_{\gamn} Y_{\ell m}(\hat{{\bm r}}_{n}) \dd\hat{\bm{r}}_{n} = 0 \qquad \forall~ 1\le n\le N, ~0\le \ell\le L, ~|m|\le\ell .
\end{eqnarray}
This indicates the weak continuity of $\partial \ul/\partial r_{n}$ and completes the proof.
\end{proof}

\subsection{Proof of Theorem \ref{them-convergence}}
\label{append-convergence}
\renewcommand{\theequation}{A.4.\arabic{equation}}
\renewcommand{\thelemma}{A.\arabic{lemma}}

\setcounter{equation}{0}  % number counter 
\setcounter{lemma}{0}

The proof of Theorem \ref{them-convergence} will be divided into two parts. First, we will show in \ref{append-convergence-a} that \eqref{convergence} holds, that is, the limiting pair $(E_{\infty},u_{\infty})$ solves the original problem \eqref{model-weak-eq}.
Then, we will justify the convergence rate \eqref{convergent-rate-eq} in \ref{append-convergence-b}.
Some auxiliary lemmas will be used, and we put their proofs in \ref{append-convergence-c}.

\subsubsection{Proof of \eqref{convergence}}
\label{append-convergence-a}

\begin{proof}
%In this proof, we use $j$ to replace $L_j$ for the subscripts, for simplicity of presentations.
%
To show that $(E_{\infty},u_{\infty})$ satisfies \eqref{model-weak-eq}, we consider that for any $v\in H^{1}(\mathbb{R}^3)$,
\begin{align}
\label{proof-3.2-a}
\nonumber
&|a(u_{\infty},v)-E_{\infty}(u_{\infty},v)| = |a_{\delta}(u_{\infty},v)-E_{\infty}(u_{\infty},v)| \\[1ex]\nonumber
\le & |a_{\delta}(u_{\infty}-u_L,v)| + |a_{\delta}(u_L,v)-E_{L}(u_L,v)| + |(E_{\infty}-E_L) (u_L,v)| + |E_{\infty} (u_L-u_{\infty},v)| \\[1ex]\nonumber
\le  & C\left( \int_{\omeout} \left|\nabla (u_{\infty}-u_L)\right| \cdot \left| \nabla v\right| + \sum_{n=1}^{N} \int_{\omen} \left| \nabla (u_{\infty}-u_L)\right| \cdot \left| \nabla v\right| + \int_{\mathbb{R}^3} \big|V (u_{\infty}-u_L) v\big| \right.
\\[1ex]
& \qquad + |a_{\delta}(u_L,v)-E_{L}(u_L,v)| +
|E_{\infty}-E_L| + \|u_L-u_{\infty}\|_{L^2(\mathbb{R}^3)} \Bigg).
\end{align}

Since \eqref{weak-convergence} and the compact embedding
\footnote{It has been shown in \cite[Lemma I.1]{lions84} that the possible loss of the compactness in the unbounded domain occurs from the vanishing of the functions, or the splitting of the functions into at least two parts which are going infinitely away from each other.
In our case, we consider functions in spaces $\ulz$ with $L\in\N$, and the ``mass'' of the functions therein concentrates around the atomic sites.
Then it is not difficult to see that the above two possibilities of loss of the compactness can be ruled out.} 
imply that $u_L$ converges to $u_{\infty}$ in $L^3(\mathbb{R}^3)$, we can derive by using the H\"{o}lder inequality that for any $v\in H^{1}(\mathbb{R}^3)$,
\begin{eqnarray}
\label{V-convergence}
\lim_{L\rightarrow \infty} \int_{\mathbb{R}^3} \big|V (u_{\infty}-u_L) v\big| \le
\lim_{L\rightarrow \infty} \|V\|_{\ltwo} \|u_L-u_{\infty}\|_{L^3(\mathbb{R}^3)} \|v\|_{L^6(\mathbb{R}^3)} = 0 .
\end{eqnarray}

Note that $u_L \in\mathcal{U}_{L}(E_L)$ satisfies \eqref{bvp-schrodinger} with the energy parameter $z=E_L$ and $\mathcal{V}_{L}(E_{L}) \subset H^{1}(\mathbb{R}^3)$, we have that for any $v\in C_{c}^{\infty}(\mathbb{R}^3)$ and $v_L\in\mathcal{V}_{L}(E_{L})$,
\begin{multline}
\label{adelta_j}
\quad
a_{\delta}(u_L,v)-E_{L} (u_L,v) = a_{\delta}(u_L,v-v_L) - E_{L} (u_L, v-v_L) + a_{\delta}(u_L, v_L)-E_{L}(u_L,v_L)
\\[1ex]
= \sum_{n=1}^{N} \int_{\gamn}  \bigg[ \frac{\partial u_L}{\partial r_{n}}\bigg]_{\Gamma_n} (v-v_L)^* 
\le C \sum_{n=1}^{N} \|v-v_L\|_{H^{\frac{1}{2}}(\gamn)} \|u_L\|_{\delta} , 
\quad
\end{multline}
where the Cauchy-Schwartz inequality and trace theorem are used for the last inequality.
Using similar arguments as that in \eqref{surface-integ-sum}-\eqref{comp-surface}, we see that for any $v\in C_{c}^{\infty}(\mathbb{R}^3)$, there exists a $v_L\in\mathcal{U}_{L}(E_L)$ such that 
$\int_{\gamn}(v-v_L)Y_{\ell m}(\hat{\bm r}_n) = 0$ for $n=1,\cdots,N,~0\le\ell\le L,~|m|\le\ell$,
and hence $\sum_{n=1}^N\|v-v_L\|_{H^{\frac{1}{2}}(\gamn)}\rightarrow 0$ as $L\rightarrow\infty$.
We then obtain from \eqref{eigenfunction-bound} and \eqref{adelta_j} that
\begin{eqnarray}
\label{convergence-lim}
\lim\limits_{L\rightarrow \infty} a_{\delta}(u_L,v)-E_{L} (u_L,v) =0 \qquad \forall~ v\in C_{c}^{\infty}(\mathbb{R}^3).
\end{eqnarray} 
Since $C_{c}^{\infty}(\mathbb{R}^3)$ is dense in $H^{1}(\mathbb{R}^3)$, \eqref{convergence-lim} also holds for any $v\in H^{1}(\mathbb{R}^3)$.

By taking into accounts \eqref{proof-3.2-a}, \eqref{eigenvalue-convergence}, \eqref{weak-convergence}, \eqref{V-convergence} and \eqref{convergence-lim}, we obtain \eqref{convergence}.
\end{proof}

\subsubsection{Proof of \eqref{convergent-rate-eq}}
\label{append-convergence-b}

We will first introduce some operators and notations, then give some lemmas that are key to the proof, and finally provide the proof of \eqref{convergent-rate-eq} by using these lemmas.
For simplicity of the notations, we will use $(E,u)$ to denote $(E_{\infty},u_{\infty})$ hereafter.

Define the solution operator $	T:\ltwo\rightarrow \hone$ such that for $f\in \ltwo$
\begin{eqnarray*}
a(Tf,v)=(f,v)\qquad\forall~v\in \hone .
\end{eqnarray*}
For given $L\in\N$ and $z\in\R$, define the approximate solution operator $T_{L}^{z}:\ltwo\rightarrow \mathcal{U}_{L}(z)$ such that for $f\in \ltwo$
\begin{eqnarray*} 
a_{L}(T^{z}_{L}f,v)=(f,\mollf v)\qquad \forall~v\in \mathcal{U}_{L}(z) .
\end{eqnarray*}
Then the eigenvalue problems \eqref{model-weak-eq} and \eqref{Galerkin-form} are equivalent to $E Tu=u$ and $E_{L} T^{E_L}_{L_L} u_{L}=u_{L}$, respectively.
For simplicity, we will use $T_L$ to denote $T^{E_L}_{L}$ in the following.
	
Let $\sigma(T)$ be the spectrum of the solution operator $T$ and $\rho(T)$ be its resolvent set.
We have that $E^{-1}\in\sigma(T)$.
Let  $\gamma$ be a circle in the complex plane that is centered at $E^{-1}$ and does not enclose any other point of $\sigma(T)$. 
Define the following two operators by contour integrals
\begin{eqnarray}\nonumber
	\mathscr{E}=\frac{1}{2\pi i}\int_{\gamma} (z-T)^{-1} \dd z 
	\quad{\rm and}\quad
	\mathscr{E}_{L}=\frac{1}{2\pi i}\int_{\gamma} (z-T_{L})^{-1} \dd z.
\end{eqnarray}
Note that $\mathscr{E}$ can be viewed as the spectral projections of $T$ related to the eigenvalue $E^{-1}$.
In fact, if $\lmax$ is sufficiently large, then the convergence $E_L\rightarrow E$ implies that $(z-T_L)^{-1}$ is well-defined and bounded on $\gamma$, and $\mathscr{E}_{L}$ is the spectral projections of $T_L$ related to $E_L^{-1}$.
We can then observe from the relationship $(z-T)^{-1}-(z-T_L)^{-1}=(z-T)^{-1}(T-T_L)(z-T_{L})^{-1}$ that
\begin{eqnarray}
\label{spectral-pro}
\|\mathscr{E}-\mathscr{E}_{L}\|_{\mathscr{L}(\hdel,\mathcal{R}(\mathscr{E}))} \le C \|T-T_L\|_{\mathscr{L}(\mathcal{R}(\mathscr{E}),\hdel)}
\end{eqnarray}
with $\mathcal{R}(\mathscr{E})$ denoting the range of $\mathscr{E}$.
	
We finally need to introduce an ``interpolation" operator from $\mathcal{R}(\mathscr{E})$ to $\mathcal{U}_{L}(z)$.
For any $u\in \mathcal{R}(\mathscr{E})$, we can write $u$ in the interstitial region $\omeout$ as
\begin{eqnarray}
\label{expand-v-eq}
u\big|_{\omeout}({\bm r}) = \sum_{n=1}^{N}\sum_{\ell m} c_{\ell m}^{n} H_{\ell m} ({\bm r}_{n};E) =: \sum_{n=1}^{N} u^{(n)}(\vr) ,
\end{eqnarray}
where $c_{\ell m}^{n}$ are the coefficients with respect to the basis $H_{\ell m}$.
We then define the interpolation operator $I_{L,z}: \mathcal{R}(\mathscr{E}) \rightarrow \mathcal{U}_{L}(z)$ by
\begin{eqnarray}
\label{interp-operator}
I_{L,z} u({\bm r}) = \sum_{n=1}^{N} \sum_{\ell\le\lmax}  c_{\ell m}^{n} \phi_{\ell m,\lmax}^n ({\bm r};z) ,
\end{eqnarray}
where $\phi_{\ell m,\lmax}^n$ are the basis functions defined in \eqref{trial-basis}.
It is easy to see from \eqref{expand-v-eq} and \eqref{interp-operator} that 
$I_{L,z} u({\bm r}) = \sum_{n=1}^{N} I_{L,z} u^{(n)}(\vr)$ for ${\bm r}\in\omeout$.

The following three lemmas are the three key ingredients of the proof.
Lemma \ref{lem-append-3} gives the error bound for interpolation $I_{L,E}$, which can be viewed as the best approximation error in the finite dimensional space $\mathcal{U}_{L}(E)$, and is crucial for the proof of the other two lemmas. 
Lemma \ref{lem-append-1} gives the error bound for the approximate solution operator $T_L$.
Lemma \ref{lem-append-2} estimates the energy error in terms of the wave functions.
We put the proofs of these lemmas in the next subsection \ref{append-convergence-c}.

\begin{lemma}
\label{lem-append-3}
Let $u\in \mathcal{R}(\mathscr{E})$. Then there exists a constant $C_s>0$ depending only on $s$ and $u$ such that 
\begin{eqnarray}
\label{interp-err}
\|u-I_{L,E}u\|_{H^r(\Omega^{\rm I})} \le C_s L^{r-s} \qquad \forall~s>r\geq 1.
\end{eqnarray}
\end{lemma}

\begin{lemma}
\label{lem-append-1}
If $u\in \mathcal{R}(\mathscr{E})$ and the asymptotic assumption \eqref{assumption} is satisfied, then there exists a constant $C_s>0$ depending only on $s$ and $u$ such that 
\begin{eqnarray}
\label{err}
\|Tu-T_L u\|_{\delta} \le C_s \left( |E-E_L| + L^{1-s} \right) .
% \|u-I_{j,E}u\|_{H^r(\Omega^{\rm I})} 
\end{eqnarray}
\end{lemma}	

\begin{lemma}
\label{lem-append-2}
Let $(E_L,u_L)$ be an eigenpair of \eqref{Galerkin-form}.
% and $\ujbar=u_j/\|u_j\|_{\ltwo}$. 
% \XL{If $\|u_j\|_{H^s(B_n)}$ is bounded uniformly with $L_j$ for any $s>1$}, 
Then there exists a constant $C_s>0$ depending only on $s$ and $u$ such that  
\begin{eqnarray}
\label{energy_bound}
|E-E_L| \le C_s \left( L^{1-s} + \|u-u_L\|^2_{\delta}\right) \qquad \forall~ s>1 .
\end{eqnarray}
\end{lemma}

By using these lemmas, we can immediately complete the proof of Theorem \ref{them-convergence}.

\begin{proof}[Proof of \eqref{convergent-rate-eq}]
Combining \eqref{spectral-pro}, \eqref{err} and \eqref{energy_bound}, we can obtain that
\begin{eqnarray}\nonumber
\|u-u_L\|_{\delta} \le C_s \left( \|u-u_L\|^2_{\delta} + L^{1-s} \right) \le C_s L^{1-s},
\end{eqnarray}
which completes the proof.
\end{proof}

\subsubsection{Proofs of the lemmas}
\label{append-convergence-c}

\begin{proof}[Proof of Lemma \ref{lem-append-3}]
The proof will be divided into two parts.
First, we will show the error on a ring by transfering the Sobolev norms into equivalent forms in the spherical coordinates.
Then we will extract the interpolation error in the interstitial region.

\newcommand{\ud}{\underline{d}}
\newcommand{\od}{\overline{d}}

\emph{Part \uppercase\expandafter{\romannumeral1}.}
Let $z<0$, $k=\sqrt{|z|}$, and $\bar{\Omega}=[\ud,\od]\times \mathbb{S}^2$ with $\od>\ud>0$. 
If $v$ satisfies the equation $-\Delta v = zv$ on $\bar{\Omega}$, then $v$ can be written in the following form
\begin{eqnarray}\nonumber
\label{u-form}
v(\bm{r};z) = \sum_{\ell m} c_{\ell m} v_{\ell}(ikr) Y_{\ell m}(\hat{\bm{r}}) 
\qquad{\rm with}\qquad 
v_{\ell}(r) = a_{\ell} j_{\ell}(r) + b_{\ell} h_{\ell}(r),
\end{eqnarray}
where the coefficients $a_{\ell}$ and $b_{\ell}$ are real numbers.
For a given $L\in\N$, let $v_L$ be the truncation of $v$ given by
\begin{eqnarray*}
v_L(\bm{r};z) = \sum_{\ell\le L} c_{\ell m} v_{\ell}(ikr) Y_{\ell m}(\hat{\bm{r}}).
\end{eqnarray*}
We will show in this first part that there exists a constant $C$ independent of $L$, such that
\begin{eqnarray}
\label{truncation-err-eq}
\|v-v_L\|_{H^r(\bar{\Omega})} \le C L^{r-s} \|v\|_{H^s(\bar{\Omega})} \quad \forall~r<s.
\end{eqnarray}

By using the fact that $\Delta=\frac{1}{r^2} \frac{\partial}{\partial r}(r^2 \frac{\partial}{\partial r}) + \frac{1}{r^2} \Delta_{S^2}$, we can transfer the Sobolev norm into the spherical coordinates as follows
\begin{align}\nonumber
\|v\|^2_{H^1(\bar{\Omega})} &= \|v\|^2_{L^2(\bar{\Omega})} + \int_{\bar{\Omega}} v(-\Delta v) + \int_{\partial\bar{\Omega}} v\frac{\partial v}{\partial n}
\\[1ex] \nonumber
&= \sum_{\ell m} |c_{\ell m}|^2 \int_{\ud}^{\od} r^2 \left(|v_\ell|^2+|z v'_\ell|^2 \right) \dd r + \sum_{\ell m} \ell(\ell+1) |c_{\ell m}|^2 \int_{\ud}^{\od} |v_{\ell}|^2 \dd r 
\\[1ex] \label{h1-norm}
&\sim \sum_{\ell m} (\ell^2+1) |c_{\ell m}|^2 \|v_{\ell}(ik\cdot)\|^2_{L^2([\ud,\od])} + \sum_{\ell m} |c_{\ell m}|^2 \| v'_{\ell}(ik\cdot)\|^2_{L^2([\ud,\od])},
\end{align}
where the notation $A \sim B$ means that there exist two positive constants $C_1$ and $C_2$ independent on $v$ such that $C_1 B\le A\le C_2 B$.
We then claim that in \eqref{h1-norm}, the second term can be controlled by the first term, more precisely, there exists $C>0$ independent of $\ell$ such that
\begin{eqnarray}
\label{radial-deriv-control-eq}
\|v'_{\ell}(ik\cdot)\|^2_{L^2([\ud,\od])} 
\le C (\ell^2+1) \|v_{\ell}(ik\cdot)\|^2_{L^2([\ud,\od])} \qquad \forall~ \ell>0.
\end{eqnarray}

To see \eqref{radial-deriv-control-eq}, we first have the following recursive formula from the properties of the functions $j_{\ell}$ and $h_{\ell}$ \cite{gonis00}
\begin{eqnarray}
\label{recursion}
v'_{\ell} = \frac{1}{2\ell+1} \left(\ell v_{\ell-1} - (\ell+1) v_{\ell+1} \right) .
\end{eqnarray}
Then we can discuss three cases respectively: (1) $v_{\ell}=h_{\ell}$; (2) $v_{\ell}=j_{\ell}$; and (3) $v_{\ell}=a_{\ell} j_{\ell} + b_{\ell} h_{\ell}$.
	
For the first case $v_{\ell}=h_{\ell}$, we have from the analytical expression of the spherical Hankel function
\begin{eqnarray}\label{analytic-hankel}
h_{\ell}(ikr) = (-i)^{\ell+2} \frac{1}{kre^{kr}} \sum_{n=0}^{\ell} \frac{1}{n! (2kr)^n} \frac{(\ell+n)!}{(\ell-n)!},
\end{eqnarray}
that ${\rm Re}\{h_{\ell}(ikr)\} \cdot {\rm Im}\{h_{\ell}(ikr)\} =0 $ for any $\ell$.
This implies
\begin{eqnarray*}\label{real-vanish}
{\rm Re} \{ h^*_{\ell+1}(ikr) h_{\ell-1}(ikr) \} = {\rm Re} \{ h_{\ell+1}(ikr) h^*_{\ell-1}(ikr) \} = 0,
\end{eqnarray*}
and hence
\begin{eqnarray}
\label{orth-hankel}
|h'_{\ell}(ikr)|^2 = \frac{\ell^2}{(2\ell+1)^2}|h_{\ell-1}(ikr)|^2 + \frac{(\ell+1)^2}{(2\ell+1)^2} |h_{\ell+1}(ikr)|^2.
\end{eqnarray}
We can derive by a direct calculation that there exist two positive constants $C_1$ and $C_2$ independent of $\ell$ such that
\begin{eqnarray}
\label{equ-hankel-eq}
C_1 |h_{\ell}(ikr)| \le |h_{\ell+1}(ikr)| 
\le C_2 \ell |h_{\ell}(ikr)| \qquad \forall~ \ell> 0, ~\ud<r<\od.
\end{eqnarray}
Combining \eqref{recursion}, \eqref{orth-hankel} and \eqref{equ-hankel-eq}, we obtain \eqref{radial-deriv-control-eq}.
	
For the second case $v_{\ell}=j_{\ell}$, we can show by similar arguments and the analytical expression of the spherical Bessel function \eqref{analytic-bessel} that
\begin{eqnarray}
\label{equ-bessel-eq}
C_1 |j_{\ell+1}(ikr)| \le |j_{\ell}(ikr)|
\le C_2 \ell |j_{\ell+1}(ikr)| \qquad \forall~ \ell> 0, ~\ud<r<\od .
\end{eqnarray}
By using \eqref{recursion} and \eqref{equ-bessel-eq}, we see that \eqref{radial-deriv-control-eq} also holds for the second case.
	
For the last case $v_{\ell}=a_{\ell} j_{\ell} + b_{\ell} h_{\ell}$,
we note that for any $\ell$, $j_{\ell}(ikr)$ and $h_{\ell}(ikr)$ are both real numbers or purely imaginary numbers at the same time.
This indicates
\begin{eqnarray*}
|v_{\ell}(ikr)|^2 = |a_{\ell}|^2 |j_{\ell}(ikr)|^2 + |b_{\ell}|^2 |h_{\ell}(ikr)|^2,
\end{eqnarray*}
and 
\begin{align*}
|v'_{\ell}(ikr)|^2 =
& |a_{\ell}|^2 
\left( \frac{\ell^2}{(2\ell+1)^2}|j_{\ell-1}(ikr)|^2 + \frac{(\ell+1)^2}{(2\ell+1)^2} |j_{\ell+1}(ikr)|^2 \right)  \\[1ex]
& + |b_{\ell}|^2  \left( \frac{\ell^2}{(2\ell+1)^2}|h_{\ell-1}(ikr)|^2 + \frac{(\ell+1)^2}{(2\ell+1)^2} |h_{\ell+1}(ikr)|^2 \right).
\end{align*}
We can then immediately obtain \eqref{radial-deriv-control-eq} by combining the first two cases.
	
By taking into accounts \eqref{h1-norm} and \eqref{radial-deriv-control-eq}, we have
\begin{eqnarray}\nonumber
\label{norm-equi-h1}
\|v\|^2_{H^1(\bar{\Omega})} \sim \sum_{\ell m} (\ell^2+1) |c_{\ell m}|^2 \left( \|a_{\ell} j_{\ell}(ik\cdot)\|^2_{L^2([\ud,\od])} + \|b_{\ell} h_{\ell}(ik\cdot)\|^2_{L^2([\ud,\od])}  \right) ,
\end{eqnarray}
which gives rise to the estimate \eqref{truncation-err-eq} with $r=1$.
We can obtain by similar arguments that
\begin{eqnarray}
\label{norm-equi-hs}
\|v\|^2_{H^s(\bar{\Omega})} \sim \sum_{\ell m} (\ell^2+1)^s |c_{\ell m}|^2 \left( \|a_{\ell} j_{\ell}(ik\cdot)\|^2_{L^2([\ud,\od])} + \|b_{\ell} h_{\ell}(ik\cdot)\|^2_{L^2([\ud,\od])}  \right) ,
\end{eqnarray}
and the estimate \eqref{truncation-err-eq} for general $r<s$ cases.
	
\vskip 0.2cm

\emph{Part \uppercase\expandafter{\romannumeral2}.}
In this second part, we will prove the estimate \eqref{interp-err} in the interstitial region.
For any $u\in R(\mathscr{E})$, it takes the form \eqref{expand-v-eq} in the interstitial region $\Omega^{\rm I}$. 
Using the truncation error \eqref{truncation-err-eq}, we have that for any $s>r$,
\begin{align}
\label{interpolation-err-eq}
\|u-I_{L,E}u\|_{H^r(\Omega^{\rm I})} 
% \le \sum_{n=1}^N \|u^n-I_{j,E}u^n\|_{H^r(\Omega^{\rm I})} 
\le \sum_{n=1}^N \|u^{(n)}-I_{L,E}u^{(n)}\|_{H^r(\mathbb{R}^3/\Omega_{n})} 
\le C L^{1-s} \sum_{n=1}^N \|u^{(n)}\|_{H^s(\mathbb{R}^3/\Omega_{n})} ,
\end{align}
where $u^{(n)}$ and $I_{L,E}$ were defined in \eqref{expand-v-eq} and \eqref{interp-operator}.
	
Note that $u$ can also be written as (see Section \ref{sec:mst})
\begin{eqnarray}
\label{v-bn-represent-eq}
u({\bm r}) = \sum_{\ell m} c_{\ell m}^{n} \left( h_{\ell}(ikr) + j_{\ell}(ikr)/t^n_{\ell}(E) \right) Y_{\ell m}(\hat{\bm{r}}_n) \qquad \vr\in B_n
\end{eqnarray}
with the $t$-matrix $t^n_{\ell}$ given by \eqref{log-deriv} or \eqref{t_l}.
%
% As a special case, the term $j_{\ell}(ikr)/t^n_{\ell}(E)$ would be dropped if $t^n_{\ell}(E)=0$ for some $\ell$.
% Note that $\chi_{\ell}$ satisfies the radial Schr\"{o}dinger equation \eqref{radial-schrodinger} and can be chosen to be real.
%
Note that \eqref{analytic-bessel} and \eqref{analytic-hankel} imply that $j_{\ell}(ikr), h_{\ell}(ikr), \sqrt{z}j'_{\ell}(ikr)$ and $\sqrt{z}h'_{\ell}(ikr)$ are all real numbers or purely imaginary numbers at the same time for any $\ell$.
Therefore, we can see from the definition \eqref{log-deriv} that $t^n_{\ell}(E)$ is real for any $\ell\ge 0$.
	
Using \eqref{v-bn-represent-eq} and the exponential decay of the spherical Hankel function, which can be seen from the analytical form \eqref{analytic-hankel}, 
we have 
\begin{align}\nonumber
& \hskip -0.3cm
\|u^{(n)}\|^2_{H^s(\mathbb{R}^3/\Omega_{n})} 
= \sum_{\ell m} (1+\ell^2)^s |c^n_{\ell m}|^2 \|h_{\ell}(ik\cdot)\|^2_{L^2([R_n,+\infty))} \\[1ex]\nonumber
&\le C\sum_{\ell m} (1+\ell^2)^s |c^n_{\ell m}|^2 \|h_{\ell}(ik\cdot)\|^2_{L^2([R_n,R_n+r^{\rm b}_n])} \\[1ex]\nonumber
&\le C \sum_{\ell m} (1+\ell^2)^s |c^n_{\ell m}|^2 \left(\|h_{\ell}(ik\cdot)\|^2_{L^2([R_n,R_n+r^{\rm b}_n])} + \|j_{\ell}(ik\cdot)/t^n_{\ell}(E)\|^2_{L^2([R_n,R_n+r^{\rm b}_n])} \right) 
\\[1ex]
\label{vn-regularity-eq}
&= C \|u\|^2_{H^s(B_n)} \le C  \|u\|^2_{H^s(\Omega^{\rm I})} ,
\end{align}
where the relation \eqref{norm-equi-hs} is used for the last equality.
Combining \eqref{interpolation-err-eq}, \eqref{vn-regularity-eq} and the regularity of $R(\mathscr{E})$, we obtain \eqref{interp-err}, which completes the proof.
\end{proof}

\begin{proof}[Proof of Lemma \ref{lem-append-1}]
Let $u\in R(\mathscr{E})$, $w=Tu$ and $w_L=T_L u$.
Then the goal is to estimate $\|w-w_L\|_{\delta}$.
Using Lemma \ref{lem-equ} and the fact that $\mathcal{V}_{L}(E_L)\subset \hone$, we have
\begin{multline*}
\qquad
\|w_L-v_L\|^2_{\delta} \le C a_{L}(w_L-v_L, w_L-v_L) = C a_{L}(w-v_L, w_L-v_L) 
\\[1ex]
\le C \|w-v_L\|_{\delta} \|\mollf(w_L-v_L)\|_{\delta} \le C \|w-v_L\|_{\delta} \|w_L-v_L\|_{\delta}
\qquad\forall~v_L\in \mathcal{U}_{L}(E_L) ,
\qquad
\end{multline*}
which leads to 
\begin{eqnarray}
\label{tri-ineqn-1}
\|w_L-v_L\|_{\delta} \le C \|w-v_L\|_{\delta} \qquad \forall~ v_L\in \mathcal{U}_{L}(E_L).
\end{eqnarray}
Using \eqref{tri-ineqn-1} and the triangle inequality
$\|w-w_L\|_{\delta} \le \|w-v_L\|_{\delta} + \|w_L-v_L\|_{\delta}$,
we have
\begin{eqnarray}
\label{quasi-opt-err}
\|w-w_L\|_{\delta} 
\le C \inf_{v_L\in\mathcal{U}_{L}(E_L)} \|w-v_L\|_{\delta}
\le C \|w-E^{-1} I_{L,E_L}u\|_{\delta} \le C \|u-I_{L,E_L}u\|_{\delta}.
\end{eqnarray}
Then we can write the right-hand side of \eqref{quasi-opt-err} as
\begin{align}
\nonumber
\|u-I_{L,E_L}u\|_{\delta} &= \|u-I_{L,E_L}u\|_{\hout} + \sum_{n=1}^{N} \|u-I_{L,E}u\|_{\hn} + \sum_{n=1}^{N} \|I_{L,E_L}u-I_{L,E}u\|_{\hn}  \\[1ex]\label{err-estimate}
&=: P_1 + P_2 + P_3 ,
\end{align}
and estimate the three terms respectively in the following.

For $P_1$, we have from \eqref{recursion}, \eqref{equ-hankel-eq} and the decay property of spherical Hankel function that
% $ |h_{\ell}(\sqrt{E} r) - h_{\ell}(\sqrt{E_j}r)| \le C |E-E_{j}|\cdot | r h'_{\ell}(\sqrt{E}r)| + \mathcal{O}(|E-E_j|^2).$
\begin{align}
\label{taylor-estimate}
\nonumber
& \hskip -0.2cm
\|I_{L,E_L}u-I_{L,E}u\|_{\hout} \le \displaystyle C |E-E_{L}| \sum_{n=1}^{N} \left\| \sum_{\ell\le\lmax}  c_{\ell m}^{n} 
|\vr_n| h'_{\ell}(\sqrt{E} |\vr_{n}|) Y_{\ell m}(\hat{\bm r}_n) \right\|_{\hout} 
\\[1ex]\nonumber
&\le C |E-E_{L}| \sum_{n=1}^{N} \sum_{\ell\le L} (1+\ell^2) |c^n_{\ell m}|^2 \|r h'_{\ell}(ikr)\|^2_{L^2([R_n,+\infty))} 
\\[1ex]\nonumber
&\le C |E-E_L| \sum_{n=1}^{N} \sum_{\ell\le L} (1+\ell^2)^2 |c^n_{\ell m}|^2 \| h_{\ell}(ikr)\|^2_{L^2([R_n,+\infty))} \\[1ex]
&\le \displaystyle C |E-E_{L}| \sum_{n=1}^{N} \|I_{L,E} u^{(n)}\|_{H^2(\omeout)},
\end{align}
where $u^{(n)}$ is given in \eqref{expand-v-eq} and the high-order terms have been omitted.
Combining the above estimate with \eqref{interp-err}, \eqref{truncation-err-eq}, \eqref{vn-regularity-eq} and the regularity of $R(\mathscr{E})$, we have
\begin{align}
\label{p1}
\nonumber
P_1 & \le \|I_{L,E_L}u-I_{L,E}u\|_{\hout} + \|u-I_{L,E}u\|_{\hout} \\[1ex]\nonumber
&\le C |E-E_{L}|  \sum_{n=1}^{N} \|I_{L,E} u^{(n)}\|_{H^2(\omeout)} + \|u-I_{L,E}u\|_{\hout} 
\\[1ex]\nonumber
&\le C |E-E_{L}|  \sum_{n=1}^{N} \left( \|u^{(n)} - I_{L,E} u^{(n)}\|_{H^2(\omeout)} + \|u^{(n)}\|_{H^2(\omeout)} \right) + \|u-I_{L,E}u\|_{\hout} 
\\[1ex]
&\le C |E-E_L| + C_s L^{1-s}  \qquad \forall~s>1.
\end{align}
	
We then estimate $P_2$ for each atomic site $1\le n\le N$.
Note that the assumption on the asymptotic problem \eqref{assumption} implies that $E$ is not an eigenvalue of the problem
\begin{eqnarray*}\label{asym-no-eigen-eq}
\left\{
\begin{array}{rl}
(-\Delta+V) u = E u &\qquad {\rm in}~ \omen, \\[1ex]
u = 0 & \qquad {\rm on}~  \gamn ,
\end{array}
\right.
\end{eqnarray*}
which together with the fact $(-\Delta + V - E) (u-I_{L,E}u) = 0$ in $\omen$ implies
\begin{eqnarray}
\label{proof-C-1}
\|u-I_{L,E}u\|_{\hn} \le C \|u-I_{L,E}u\|_{H^{\frac{1}{2}}(\gamn)} .
\end{eqnarray}
Then by using \eqref{interp-err} and \eqref{proof-C-1}, we have
\begin{align}
\label{I-3}
\nonumber
\|u-I_{L,E}u\|_{\hn} & \le C \bigg( \big\|[I_{L,E}u]_{\gamn}\big\|_{H^{\frac{1}{2}}(\gamn)} +  \|u-(I_{L,E}u)^-\|_{H^{\frac{1}{2}}(\gamn)} \bigg) 
\\[1ex]\nonumber
& \le C \bigg( L^{1-s} \big\|I_{L,E}u\big\|_{H^{s}(\Omega^{\rm I})} +  \|u-I_{L,E}u\|_{\hout} \bigg) 
\\[1ex]
& \le C L^{1-s} \bigg( \sum_{n=1}^N \big\|I_{L,E}u^{(n)}\big\|_{H^{s}(\Omega^{\rm I})} +  \|u\|_{H^s(\omeout)} \bigg) . 
\qquad
\end{align}
Therefore, it follows from \eqref{vn-regularity-eq} and the regularity of $R(\mathscr{E})$ that there exists a $C_s>0$ depending only on $u$ and $s$ such that
\begin{eqnarray}
\label{p2}
P_2 \le C_s L^{1-s} \qquad \forall~s>1.
\end{eqnarray}
	
To estimate the last term $P_3$, we have that in each atomic sphere $\omen$,
\begin{eqnarray*}
(-\Delta + V - E_L) (I_{L,E}u-I_{L,E_L}u)= (E-E_L) I_{L,E}u ,
\end{eqnarray*}
which together with \eqref{taylor-estimate}, \eqref{I-3} and the weak continuity of $I_{L,z}v$ (through $\Gamma_n$)
implies that
\begin{align}
\nonumber
& \hskip 0.2cm
\|I_{L,E}u-I_{L,E_L}u\|_{\hn} 
\le C \left( |E-E_L| \|I_{L,E}u\|_{L^2(\omen)} + \|(I_{L,E}u)^+-(I_{L,E_L}u)^+\|_{H^{\frac{1}{2}}(\gamn)} \right) 
\\[1ex] \nonumber
& \le C \left( |E-E_L| \|I_{L,E}u\|_{L^2(\omen)} + \|(I_{L,E}u)^--(I_{L,E_L}u)^-\|_{H^{\frac{1}{2}}(\gamn)} \right) 
\\[1ex] \nonumber
& \le C |E-E_L| \left( \|u\|_{L^2(\omen)} +  \|u-I_{L,E}u\|_{L^2(\omen)} \right)  + C \|I_{L,E}u-I_{L,E_L}u\|_{H^{1}(\Omega^{\rm I})}  
\\[1ex]\nonumber %\label{I-2}
& \le C |E-E_L| \left( \|u\|_{L^2(\omen)} +  \big\|[I_{L,E}u]_{\Gamma_n} \big\|_{H^{\frac{1}{2}}(\gamn)} + \|u-I_{L,E}u\|_{H^1(\omeout)} + \sum_{n=1}^{N} \|I_{L,E} u^{(n)}\|_{H^2(\omeout)} \right).
\end{align}
It follows from \eqref{interp-err}, \eqref{vn-regularity-eq} and the regularity of $R(\mathscr{E})$ that there exists $C_s>0$ depending only on $u$ and $s$ such that
\begin{eqnarray}
\label{p3}
P_3 \le C_s |E-E_L|.
\end{eqnarray}

Taking into accounts \eqref{quasi-opt-err}, \eqref{err-estimate}, \eqref{p1}, \eqref{p2} and \eqref{p3}, we complete the proof of \eqref{err}.
\end{proof}

\begin{proof}[Proof of Lemma \ref{lem-append-2}]
Let $\bar{u}_L= u_L / \|u_L\|_{\ltwo}$.
Since $(E, u)$ and $(E_L, u_L)$ satisfy \eqref{model-weak-eq} and \eqref{Galerkin-form}, respectively, we have
\begin{align}
\label{proof-C-Q}
\nonumber
& \hskip -0.2cm
E-E_L = a_{\delta}(u,u) - a_{\delta}(u_L,\mollf u_L)
\\[1ex]\nonumber
& = \big( a_{\delta}(u,u) - a_{\delta}(\bar{u}_L,\bar{u}_L) \big) 
+ \big( a_{\delta}(\bar{u}_L,\bar{u}_L)-a_{\delta}(u_L,u_L) \big) 
+ \big( a_{\delta}(u_L,u_L) - a_{\delta}(u_L,\mollf u_L) \big) 
\\[1ex]
& =: Q_1 + Q_2 + Q_3 .
\end{align}

For the first term $Q_1$, we have from \eqref{model-weak-eq} and \eqref{Galerkin-form} that
\begin{align}
\label{q1}
\nonumber
Q_1 &= a_{\delta}(u,u)-a_{\delta}(\ujbar,\ujbar) = -a_{\delta}(u-\ujbar,u-\ujbar)+2a_{\delta}(u,u-\ujbar)
\\[1ex]\nonumber
&= -a_{\delta}(u-\ujbar,u-\ujbar) + 2E(u,u-\ujbar) + 2D_{\delta}
\\[1ex]
&= -a_{\delta}(u-\ujbar,u-\ujbar) + E (u-\ujbar,u-\ujbar) + 2D_{\delta} ,
\end{align}
where $D_{\delta}$ is the consistency error, and can be estimated by Lemma \ref{lem-equ} (b2) and the weak continuity of $u_L$ as
\begin{align}
\label{consistency-err}
\nonumber
D_{\delta} & = a_{\delta}(u,u-\ujbar) - E(u,u-\ujbar) = -a_{\delta}(u,\ujbar) + E(u,\ujbar) 
\\[1ex]
& \le C \|u\|_{\hone} \sum_{n=1}^{N} \|[u_L]_{\Gamma_n}\|_{H^{\frac{1}{2}}(\gamn)} 
\le C L^{1-s} \|u\|_{\hone}    \sum_{n=1}^{N} \|u_L\|_{H^{s}(B_n)} .
\end{align}
Note that $u_L$ solves the strong form of the eigenvalue problem in the interstitial region, we can deduce that $\sum_{n=1}^{N}\|u_L\|_{H^s(B_n)}$ is uniformly bounded for any $L$.
Then by using \eqref{eigenfunction-bound}, \eqref{truncation-err-eq}, $(u_L,\mollf u_L)=1$ and the fact that $\mollf$ only mollifies the high-frequency components (i.e., $\ell>L$) of the function in $B_n~(n=1,\cdots,N)$, we have
\begin{multline}
\label{proof-M-error}
\|u_L-\ujbar\|_{\delta} = \left| \frac{1}{\|u_L\|_{\ltwo}} -1 \right| \|u_L\|_{\delta}
=  \frac{\left|(u_L,u_L)^{1/2}-(u_L,\mollf u_L)^{1/2}\right|}{\|u_L\|_{\ltwo}} \|u_L\|_{\delta} \\[1ex]
=  \frac{\left|(u_L, u_L-\mollf u_L)\right|}{\|u_L\|_{\ltwo}(\|u_L\|_{\ltwo}+1)}  \|u_L\|_{\delta} 
\le C 
\|u_L - \mollf u_L\|_{\ltwo}
\le C_s \lmax^{-s} \quad \forall~s>1.
\end{multline}
By combining \eqref{q1}, \eqref{consistency-err} and \eqref{proof-M-error}, we obtain
\begin{eqnarray}
\label{proof-C-Q1}
Q_1 \le C_s L^{-s} + C \|u-\ujbar\|^2_{\delta}
\le C_s L^{-s} + C \|u-u_L\|^2_{\delta} \qquad \forall~s>1.
\end{eqnarray}

For $Q_2$, by using the same technique in \eqref{proof-M-error} we have that
%\begin{multline}
%\label{proof-C-Q2}
%\qquad
%Q_2 = \left( \frac{1}{\|u_j\|^2_{\ltwo}} -1 \right) a_{\delta}(u_j,u_j)
%= \frac{(u_j, \mollf u_j - u_j)}{\|u_j\|^2_{\ltwo}} a_{\delta}(u_j,u_j)
%\\[1ex]
%% \le C \frac{\|u_j - \mollf u_j\|_{\ltwo}}{\|u_j\|_{\ltwo}}
%\le C \|u_j - \mollf u_j\|_{\ltwo} \le C L_j^{-s} \left(\sum_j^N\|u_j\|_{H^s(B_n)}\right) .
%\qquad
%\end{multline}
\begin{eqnarray}\label{proof-C-Q2}
Q_2 \le	C_s L^{-s} \qquad \forall~s>1.
\end{eqnarray}
For the final term $Q_3$, we can estimate by similar arguments that
\begin{eqnarray}
\label{proof-C-Q3}
Q_3 \le C \|u_L\|_{\delta} \|u_L- \mollf u_L\|_{\delta} 
\le C \|u_L- \mollf u_L\|_{\delta} \le C_s L^{1-s} \qquad \forall~s>1.
%\le C \|\mollf-\mathcal{I}\|_{\mathscr{L}(\mathcal{U}_{j}(E_j),\hdel)}.
\end{eqnarray}

By taking into accounts the estimates \eqref{proof-C-1}, \eqref{proof-C-Q1}, \eqref{proof-C-Q2} and \eqref{proof-C-Q3}, we complete the proof of \eqref{energy_bound}.
\end{proof}

% appendix LS equation

\section{Lippmann-Schwinger equation}
\label{append-free-electron-Green}
\renewcommand{\theequation}{B.\arabic{equation}}
\setcounter{equation}{0}  

In this appendix, we will revisit the Lippmann-Schwinger equation \cite{gonis00} and show an alternative (physical) definition of the $t$-matrix.

Let $G_{0}$ be the Green's function of the free-electron system with energy $z$, which is given by the solution of the following problem 
\begin{eqnarray}
\label{green-def}
\left\{
\displaystyle
\begin{array}{ll}
(z+\Delta) G_{0}(\bm{r},\bm{r}';z)= \delta(\bm{r}-\bm{r}') &\qquad \bm{r},\bm{r}' \in\mathbb{R}^3, \\[1ex]
G_{0}(\bm{r},{\bm r}';z) \rightarrow 0 &\qquad |\bm{r}-\bm{r}'|\rightarrow \infty .
\end{array}\right.
\end{eqnarray}
Then one observes that the solution $\psi$ of \eqref{single-schrodinger} can be written as
\begin{eqnarray}
\label{lipp-sch}
\psi({\bm r}) = \psi_0(\bm{r}) + \int_{\mathbb{R}^3} G_{0}(\bm{r},\bm{r}';z) V_{0}(r') \psi({\bm r}') \dd {\bm r}' \qquad \vr\in\mathbb{R}^3,
\end{eqnarray}
where $\psi_0$ satisfies $-\Delta \psi_0 = z\psi_0$.
This is the well-known Lippmann-Schwinger equation in the quantum scattering theory.

To decompose \eqref{lipp-sch} into each angular-momentum component, we will first need an angular-momentum representation of the free electron Green's function 
\begin{eqnarray}
\label{free-green-angular-momentum}
G_{0}(\bm{r},\bm{r'};z)=\sumlm \ylm g_{\ell}(r,r';z) Y_{\ell m}(\hat{\bm{r}}'),
\end{eqnarray}
where $g_{\ell}$ is given by
\begin{eqnarray}
\label{g_l-expansion}
g_{\ell}(r,r';z)=-i\sqrt{z}j_{\ell}(\sqrt{z}r_<)h_{\ell}(\sqrt{z}r_>)
\end{eqnarray}
with $r_>={\rm max}(r,r')$ and $r_<={\rm min}(r,r')$.
Although the expression \eqref{free-green-angular-momentum} has been well known and widely used in physics, we do not find an appropriate literature for a complete proof.
Therefore, we will provide a proof of \eqref{free-green-angular-momentum} in the following.
By using the angular momentum representation of $\psi$ in \eqref{radial-representation}, the symmetry of $V$, and \eqref{free-green-angular-momentum}, we can rewrite \eqref{lipp-sch} by
\begin{eqnarray}
\label{radial-eq}
\chi_{\ell}(r;z)=j_{\ell}(\sqrt{z}r)+\int_{0}^{R_0} g_{\ell}(r,r';z) V_0(r') \chi_{\ell}(r';z) r'^2 \dd r' \qquad \forall~\ell \in\N .
\end{eqnarray}
Combining \eqref{g_l-expansion} and \eqref{radial-eq} can lead to 
\begin{eqnarray}
\label{radial-general-LS}
\chi_{\ell}(r;z) = j_{\ell}(\sqrt{z}r) + t_{\ell}(z) h_{\ell}(\sqrt{z}r) \qquad r\geq R_0 .
\end{eqnarray}
with $t_l$ given by \eqref{t_l}.
Note that \eqref{chi_l} and \eqref{radial-general-LS} only differ by a normalization constant.

\begin{proof}[Proof of \eqref{free-green-angular-momentum}]

Using \eqref{green-def}, the identity
\begin{align*}
\delta(\bm{r}-\bm{r}')
= \frac{1}{r^2}\delta(r-r')\delta(\bm{\hat{r}}-\hat{\bm{r}}')
= \frac{2}{\pi} \sum_{\ell m} Y_{\ell m}(\bm{\hat{r}}) Y_{\ell m}(\hat{\bm{r}}') 
\int_{0}^{+\infty} \lambda^2 j_{\ell}(\lambda r)  j_{\ell}(\lambda r') \dd\lambda ,
\end{align*}
and the fact that the spherical Bessel functions $j_{\ell}$ solve
\begin{eqnarray}
(\Delta+z) j_{\ell}(\lambda r) \ylm = (z-\lambda^2) j_{\ell}(\lambda r) \ylm 
\qquad{\rm for}~z\in\mathbb{C},~z\notin\{0\}\cup\mathbb{R}^+ ,
\end{eqnarray}
we have
\begin{eqnarray}
\label{proof-a}
G_{0}(\bm{r},\bm{r}';z)
= \frac{2}{\pi}\sum_{\ell m} Y_{\ell m}(\bm{\hat{r}}) Y_{\ell m}(\hat{\bm{r}}') \int_{0}^{+\infty} \frac{\lambda^2}{z-\lambda^2} j_{\ell}(\lambda r) j_{\ell}(\lambda r') \dd\lambda .
\end{eqnarray}
From the odevity of spherical Bessel and Newmann functions, i.e., $j_{\ell}(-x)=(-1)^{\ell} j_{\ell}(x)$ and $n_{\ell}(-x)=(-1)^{\ell+1} n_{\ell}(x)$, we have
\begin{multline}
\label{appen-c-gl}
\frac{2}{\pi} \int_{0}^{+\infty} \frac{\lambda^2}{z-\lambda^2} j_{\ell}(\lambda r) j_{\ell}(\lambda r') \dd\lambda 
= -\frac{1}{\pi} \int_{-\infty}^{+\infty} \frac{\lambda^2}{(\lambda-\sqrt{z})(\lambda+\sqrt{z})} j_{\ell}(\lambda r) j_{\ell}(\lambda r') \dd\lambda \\[1ex]
= -\frac{1}{\pi} \int_{-\infty}^{+\infty} \frac{\lambda^2}{(\lambda-\sqrt{z})(\lambda+\sqrt{z})} j_{\ell}(\lambda r_<) h_{\ell}(\lambda r_>) \dd\lambda
=: \int_{-\infty}^{+\infty} f(\lambda) \dd\lambda .
\qquad
\end{multline}

In order to evaluate \eqref{appen-c-gl}, we extend the integrand $f(\lambda)$ to the complex plane, and introduce an infinite semicircle $S_{\infty}$ in the upper plane, with 
$S_{\infty} := \lim\limits_{|x|\rightarrow +\infty} S_{|x|}$ and $S_{|x|}=\{x=|x|e^{i\theta},~0\le\theta\le\pi\}$,
such that the real axis and $S_{\infty}$ form a closed contour.
Since $f(\lambda)$ has only one singular point $\lambda=\sqrt{z}$ inside the contour, the residual theorem yields
\begin{align}
\label{proof-c}
\int_{-\infty}^{+\infty} f(\lambda) \dd\lambda 
= 2\pi i{\rm Res} f(\sqrt{z}) - \int_{S_{\infty}} f(\lambda) \dd\lambda  
= -i\sqrt{z} j_{\ell}(\sqrt{z}r_<)  h_{\ell}(\sqrt{z}r_>) - \int_{S_{\infty}} f(\lambda) \dd\lambda.
\end{align}
From the asymptotic behaviors of the spherical Bessel and Hankel functions as $|x|\rightarrow +\infty$
\begin{eqnarray*}
h_{\ell}(x) \rightarrow (-i)^{\ell+1}~ \frac{e^{ix}}{x} \quad \text{and} \quad
j_{\ell}(x) \rightarrow ~i^{\ell} ~\frac{\sin x}{x},
\end{eqnarray*} 
we have $\displaystyle \lim_{|x|\rightarrow +\infty} \int_{S_{|x|}} f(\lambda)\dd \lambda=0$.
This together with \eqref{proof-a}, \eqref{appen-c-gl} and \eqref{proof-c} completes the proof of \eqref{free-green-angular-momentum}.
\end{proof}

At the end of this appendix, we provide a proof to show that the two definitions of the $t$-matrix \eqref{log-deriv} and \eqref{t_l} are equivalent.

\begin{proof}[Proof of the equivalence of  \eqref{log-deriv} and \eqref{t_l}]

First let $t_{\ell}$ be given by \eqref{t_l}, we have from \eqref{radial-eq} and \eqref{g_l-expansion} that
\begin{eqnarray}
\label{proof-d}
\chi_{\ell}(r;z)=\left\{
\begin{array}{ll}
j_{\ell}(\sqrt{z}r) + t_{\ell}(z) h_{\ell}(\sqrt{z}r) &\quad {\rm if}~r>R_0, \\[1ex]
\displaystyle
j_{\ell}(\sqrt{z}r) -i\sqrt{z} h_{\ell}(\sqrt{z}r) \int_{0}^{r} j_{\ell}(\sqrt{z}r') V_0(r') \chi_{\ell}(r';z) r'^2 \dd r'  \\[1ex]
\displaystyle
\qquad\qquad -i\sqrt{z}  j_{\ell}(\sqrt{z}r) \int_{r}^{R_0}  h_{\ell}(\sqrt{z}r') V_0(r') \chi_{\ell}(r';z) r'^2 \dd r'  &\quad {\rm if}~0<r\le R_0 .
\end{array}\right.
\end{eqnarray}
We observe from \eqref{proof-d} that 
\begin{align*}
& \chi_{\ell}(R_0;z) := \lim_{r\rightarrow R^{-}_{0}} \chi_{\ell}(r;z) 
= j_{\ell}(\sqrt{z}R_0) + t_{\ell}(z) h_{\ell}(\sqrt{z}R_0)
= \lim_{r\rightarrow R^{+}_{0}} \chi_{\ell}(r;z) 
\qquad\qquad{\rm and}
\\[1ex]
& \chi'_{\ell}(R_0;z) := \lim_{r\rightarrow R^{-}_{0}} \chi'_{\ell}(r;z) 
= \sqrt{z} j'_{\ell}(\sqrt{z}R_0) + \sqrt{z} t_{\ell}(z) h'_{\ell}(\sqrt{z}R_0)
= \lim_{r\rightarrow R^{+}_{0}} \chi'_{\ell}(r;z).
\end{align*}
The above two equations lead to
\begin{eqnarray*}
\sqrt{z} \chi_{\ell}(R_0;z) \left(  j'_{\ell}(\sqrt{z}R_0) + t_{\ell}(z) h'_{\ell}(\sqrt{z}R_0) \right)
= \chi'_{\ell}(R_0;z) \left( j_{\ell}(\sqrt{z}R_0) + t_{\ell}(z) h_{\ell}(\sqrt{z}R_0) \right),
\end{eqnarray*}
which together with a direct calculation implies that $t_{\ell}$ equals \eqref{log-deriv}.

Conversely, if $t_{\ell}$ be given by \eqref{log-deriv}. 
One can derive that $t_{\ell}$ equals \eqref{t_l} by using \eqref{radial-eq} and similar arguments.
\end{proof}

% sec: non-spherical potentials

\section{Formalism for non-spherical potentials}
\label{append-asym}
\renewcommand{\theequation}{C.\arabic{equation}}
\setcounter{equation}{0}

In this appendix, we will generalize the formalism to problems with non-spherically symmetric potentials. We will still require the potentials to be supported by non-overlapping atomic spheres.
Let us consider the single-site problem with $V_0$ in \eqref{single-schrodinger} given by
\begin{eqnarray}\nonumber
\label{non-spherical-potential}
V_{0}(\bm r)=\sum_{\ell m} V_{\ell m}(r) Y_{\ell m}(\bm{\hat{r}}),
\end{eqnarray}
with $V_{\ell m}(r)=0,~\forall~r>R_0$ and $\ell\in\N,~|m|\leq\ell$.
Compared to the problem with spherically symmetric potentials, the major difference is that an incoming wave cannot preserve its angular momentum after being scattered by the non-spherical potential $V_0$.
More precisely, if we consider the Lippmann-Schwinger equation \eqref{lipp-sch} with an incoming wave $\psi_0(\vr)=J_{\ell m}(\bm{r};z)$ for some $\ell\in\N$ and $|m|\le\ell$, then the corresponding outgoing wave can be written as
\begin{eqnarray}
\label{lipp-full}
\zeta_{\ell m}({\bm r};z) = J_{\ell m}(\bm{r};z) + \int_{\mathbb{R}^3} G_0({\bm r},{\bm r'};z) V_0(\bm r') \zeta_{\ell m}({\bm r'};z) \dd{\bm r'} \qquad
\bm{r}\in\mathbb{R}^3.
\end{eqnarray}
It is important to note that $\zeta_{\ell m}$ contains more angular momentum than $Y_{\ell m}$ since $V_0$ is not spherically symmetric, and the subscript of $\zeta_{\ell m}$ only indicates the angular momentum of its incoming wave.

Note that $\zeta_{\ell m}$ satisfies the Schr\"{o}dinger equation \eqref{schrodinger} for any $\ell\in\N$ and $|m|\le\ell$.
We can write $\zeta_{\ell m}$ in the angular momentum representation as
\begin{eqnarray}\nonumber
\zeta_{\ell m}(\bm{r};z) =\sum_{\ell' m'} \chi_{\ell' m',\ell m}(r;z) Y_{\ell' m'}(\bm{\hat{r}}),
\end{eqnarray}
where $\chi_{\ell' m',\ell m}$ satisfies a system of coupled equations in the atomic sphere ($r<R_0$) for any fixed $\ell\in\N$ and $|m|\le\ell$
\begin{eqnarray}
\label{radial-nonsym}
\sum_{\ell'm'} \left[ \left( -\frac{1}{r}\frac{\partial^2}{\partial r^2} r + \frac{\ell'(\ell'+1)}{r^2} -z \right) \delta_{\ell'\ell''} \delta_{m'm''} + \overline{V}_{\ell' m', \ell''m''}(r)  \right] \chi_{\ell'm',\ell m}(r;z) = 0
\end{eqnarray}
with the potential 
$\overline{V}_{\ell' m', \ell''m''}(r) = \sum_{\ell'''m'''} C_{\ell'm',\ell''m'',\ell''' m'''} V_{\ell'''m'''}(r)$.

Then we can rewrite \eqref{lipp-full} in the interstitial region by
\begin{eqnarray}\nonumber
\zeta_{\ell m}({\bm r};z) = J_{\ell m}(\bm{r};z) 
+ \sum_{\ell' m'} t_{\ell' m', \ell m}(z) H_{\ell' m'}(\bm{r};z) \qquad
{\rm for}~\vr\in\omeout,
\end{eqnarray}
where $t_{\ell' m', \ell m}(z)$ are matrix elements of the $t$-matrix (an $(L+1)^2\times(L+1)^2$ matrix)
\begin{align}
\label{t-matrix-nonspherical-}
\nonumber
t_{\ell' m', \ell m}(z) 
&= -i\sqrt{z} \int_{\mathbb{R}^3} J_{\ell' m'}({\bm r'};z) V_0(\bm r') \zeta_{\ell m}({\bm r'};z) \dd{\bm r'} 
\\[1ex] \nonumber
&= -i\sqrt{z}\int_{0}^{R_0} j_{\ell'}(\sqrt{z}r') \sum_{\ell''m''} \overline{V}_{\ell' m', \ell''m''}(r') \chi_{\ell'' m'',\ell m}(r';z) r'^2 \dd r'.
\end{align}
Compared to \eqref{t_l}, we see that the $t$-matrix for a non-spherical potential is no longer diagonal.

Let $a^n_{\ell m}$ and $b^n_{\ell m}$ be the coefficients as those in \eqref{single-wavefun}, then we have the following relationship which is similar to \eqref{multi-incoming-scattering}
\begin{eqnarray}
\label{full-in-sc-coef}
b^n_{\ell m} =  \sum_{\ell' m'} t^n_{\ell m,\ell'm'}(z) a^n_{\ell' m'} .
\end{eqnarray}
Combining \eqref{coef-in-sc} and \eqref{full-in-sc-coef} can give the equation for systems with non-spherical potentials
\begin{eqnarray}
\label{kkr-coeff-eq-full}
\nonumber
\sum_{n=1}^{N} \sum_{\ell m} \left( \delta_{nn'} \delta_{\ell\ell'} \delta_{mm'}-  \sum_{\ell''m''} t^n_{\ell''m'',\ell m}(E) g^{nn'}_{\ell''m'',\ell' m'}(E) \right)  a^{n}_{\ell m} =0 \qquad\qquad 
\\[1ex]
{\rm for}~ n' = 1,\cdots,N,~\ell'=0,1,\cdots,~{\rm and}~ |m'|\leq\ell'.
\end{eqnarray}
Then, for a given angular momentum cut-off $L$, we have the KKR secular equation ${\rm Det} \big( S(E_L) \big) = 0$,
where the matrix elements of $S(E_L)\in \mathbb{C}^{(N(L+1)^2)\times (N(L+1)^2)}$ are given by the coefficients of the system \eqref{kkr-coeff-eq-full}.

Once the approximate solution $E_L$ is obtained, the corresponding wave function can be given by
\begin{eqnarray}
\label{kkr-wave-full-eq}
u_{\lmax}(\bm r;E_{\lmax}) = \left\{
\begin{array}{ll}
\displaystyle \sum_{\ell\le \lmax} a^{n}_{\ell m}(E_{\lmax}) \zeta^{n}_{\ell m}(\bm{r}_{n};E_{\lmax})  
&\quad \bm{r}\in \omen~(n=1,\cdots, N), 
\\[1ex]
\displaystyle \sumn \sum_{\ell \le \lmax} b^{n'}_{\ell m}(E_{\lmax}) H_{\ell m}(\bm{r}_{n'};E_{\lmax}) 
&\quad \bm{r}\in \Omega^{\rm I},
\end{array}
\right.
\end{eqnarray}
where the coefficients $\pmb{a}:=\{a^{n}_{\ell m}(E_{\lmax})\}$ satisfy $S(E_L) \pmb{a} = 0$ and $\zeta^{n}_{\ell m}$ in each atomic sphere $\Omega_n$ is
\begin{eqnarray}
\nonumber
\zeta^n_{\ell m}(\bm{r}_{n};E_L) = \sum_{\ell'\le L} \chi^n_{\ell' m',\ell m}(r_{n};E_L) Y_{\ell' m'}(\bm{\hat{r}}_n) \qquad |\vr_n|<R_n,
\end{eqnarray}
with $\chi_{\ell' m',\ell m}^n$ the solution of \eqref{radial-nonsym} satisfying the boundary condition
\begin{eqnarray}\nonumber
\chi_{\ell' m',\ell m}^n(R_n;E_{\lmax}) = j_{\ell}(\sqrt{E_{\lmax}}R_n) \delta_{\ell\ell'} \delta_{mm'} + t^{n}_{\ell'm',\ell m}(E_{\lmax}) h_{\ell'}(\sqrt{E_{\lmax}}R_n).
\end{eqnarray}

A key observation is that the wave function $u_L$ given in \eqref{kkr-wave-full-eq} is weakly continuous through the atomic spherical surfaces $\Gamma_n~(n=1,\cdots,N)$, as well as its first derivative.
Then with the above formalism, all our analysis for problems with spherical potentials can be generalized directly to non-spherical problems.

\small
\bibliographystyle{plain}
\bibliography{bib}

\end{document}